\documentclass[11pt]{amsart}
\usepackage{amsthm}
\usepackage{amsmath,amssymb,amscd,eucal}
 \usepackage{latexsym}
\usepackage{graphicx}
\usepackage[usenames,dvipsnames,svgnames,table]{xcolor}
\usepackage{tikz-cd}
\usepackage{color} 
\usepackage{cancel}
\usepackage{enumitem}
\usepackage[normalem]{ulem}

\newtheorem{thm}[equation]{Theorem}
\newtheorem{cor}[equation]{Corollary}
\newtheorem{lemma}[equation]{Lemma}
\newtheorem{prop}[equation]{Proposition}

\newtheorem*{conj*}{Conjecture}

\newtheorem*{thmA}{Theorem A}
\newtheorem*{thmB}{Theorem B}
\newtheorem*{thmC}{Theorem C}

\newtheorem{remark}[equation]{Remark}

\newtheorem{exam}[equation]{Example}
\newtheorem{exams}[equation]{Examples}

\numberwithin{equation}{section}

\renewcommand{\ker}{\mathsf{ker}}
\renewcommand{\gcd}{\mathsf{gcd}}

\newcommand{\FF}{\mathbb{F}}  
  
\newcommand{\ZZ}{\mathbb{Z}}  
\newcommand{\CC}{\mathbb{C}}  
 
\newcommand{\KK}{\mathbb{K}} 

\newcommand{\A}{\mathsf{A}} 
\newcommand{\BB}{\mathsf{B}}  
\newcommand{\R}{\mathsf{R}} 
\newcommand{\D}{\mathsf{D}}

\newcommand{\V}{\mathsf{V}} 
\newcommand{\W}{\mathsf{W}}

\newcommand{\pr}{\mathsf{u}}

\newcommand\ot{\otimes}
\newcommand\y{\hat y}

\newcommand\Hom{\mathsf{Hom}}

\newcommand\chara{\mathsf {char}}

\newcommand\dd{\mathsf{d}}
\newcommand\s{\mathsf{s}}
\newcommand\rD{\phi}
\newcommand\rr{\mathsf{r}}
\newcommand{\gb}[1]{\left[ #1\right]}
\newcommand{\barr}[1]{\overline{#1}}
\newcommand{\eval}[2]{\left. #1 \right|_{#2}}

\newcommand\inder{ \mathsf {Inder_\FF}}
\newcommand\der{ \mathsf {Der_\FF}}
\newcommand\hoch{\mathsf{HH}}
\newcommand\ad{\mathsf {ad}}
\newcommand\im{\mathsf{im}}
 
\newcommand{\cent}[1]{\mathsf{Z}(\A_{#1})}
\newcommand\modd[1]{\ (\mathsf{mod} \, #1)}
\newcommand\degg{\, \mathsf{deg} \,}

\newcommand\spann{\, \mathsf{span}}

\allowdisplaybreaks
 
\begin{document}

\title[Lie structure on the Hochschild cohomology of $\A_h$]{Lie structure on the Hochschild cohomology of a family of subalgebras of the Weyl algebra} 

\date{} 

\author{Samuel A.\ Lopes$^*$ and Andrea Solotar}

\thanks{$^*$ The author was partially supported by CMUP (UID/MAT/00144/2019), which is funded by FCT with national (MCTES) and European structural funds through the programs FEDER, under the partnership agreement PT2020.\newline  \newline
\textbf{MSC Numbers (2010)}: 16E40, 16W25, 16S32, 16S36 \hfill \newline
\textbf{Keywords}: Weyl algebra, Ore extension, Hochschild cohomology, Gerstenhaber bracket, Witt algebra}

\address{CMUP, Faculdade de Ci\^encias\\
Universidade do Porto, Rua do Campo Alegre 687 \\
4169-007 Porto, Portugal}
\email{slopes@fc.up.pt}
 
\address{IMAS and Dto.\ de Matem\'atica, Facultad de Ciencias Exactas y Naturales\\
Universidad de Buenos Aires, Ciudad Universitaria, Pabell\'on 1 \\
(1428) Buenos Aires, Argentina}
\email{asolotar@dm.uba.ar}

\vspace{-.25 truein}  
\begin{abstract} 
For each nonzero $h\in \FF[x]$, where $\FF$ is a field, let $\A_h$ be the unital associative algebra generated by elements $x,y$,  satisfying the relation $yx-xy = h$. This gives a parametric family of subalgebras of the Weyl algebra $\A_1$, containing many well-known algebras which have previously been studied independently. In this paper, we give a full description the Hochschild cohomology $\hoch^\bullet(\A_h)$ over a field of arbitrary characteristic. In case $\FF$ has positive characteristic, the center $\cent{h}$ of $\A_h$ is nontrivial and we describe $\hoch^\bullet(\A_h)$ as a module over $\cent{h}$. The most interesting results occur when $\FF$ has characteristic $0$. In this case, we describe $\hoch^\bullet(\A_h)$ as a module over the Lie algebra $\hoch^1(\A_h)$ and find that this action is closely related to the intermediate series modules over the Virasoro algebra. We also determine when $\hoch^\bullet(\A_h)$ is a semisimple $\hoch^1(\A)$-module.  
\end{abstract}

\maketitle

\section{Introduction}\label{S:intro}
Given a field $\FF$ and a nonzero polynomial $h(x) \in \FF[x]$, let $\A_h$ be the unital associative $\FF$-algebra with two generators $x$ and $\hat{y}$, subject to the relation $\y x-x\y =h$. The aim of this article is to describe the structure---given by the Gerstenhaber bracket---of the Hochschild cohomology spaces $\hoch^\bullet(\A_h)$  as Lie modules over $\hoch^1(\A_h)$.

The family $\A_h$ parametrizes many well-known algebras, which we study simultaneously. For $h=1$, we retrieve the first Weyl algebra $\A_1$. Other particular cases have attracted attention, such as
$\A_x$, which is the universal enveloping algebra of the two-dimensional non-abelian Lie algebra and $\A_{x^2}$, known as the \textit{Jordan plane}, which is a Nichols algebra of non diagonal type. More generally, taking $h=x^n$ with $n\geq 3$ and setting $x$ in degree $1$ and $\y$ in degree $n-1$ then, as observed by Stephenson \cite{dS96}, $\A_{x^n}$ is Artin-Schelter regular of global dimension two, although it does not admit any regrading so that it becomes generated in degree one.

It is well known that the Weyl algebra is the algebra of differential operators over the one dimensional affine space, where $x$ acts by multiplication and $y$ corresponds to the usual derivative $\frac{\partial}{\partial x}$. Of course, replacing this last action by 
$h\cdot\frac{\partial}{\partial x}$ for any fixed polynomial $h(x) \in \FF[x]$ also corresponds to a derivation. If $h=0$, the derivation would annihilate everywhere, so we will not consider this case.

There is an embedding of $\A_h$ in $\A_1$ given by 
$x\mapsto x$, $\y\mapsto yh$, as in \cite[Lem.\ 3.1]{BLO15tams}. We will thus henceforth take $\y=yh$ and consider $\A_h$ as the unital subalgebra of the Weyl algebra $\A_1$ generated by $x$ and $\y=yh$, where $[y, x]=1$ and $[\y, x]=h$.

The paper is organized as follows. In Section \ref{S:comm} we prove a few technical lemmas about commutators, while in Section \ref{S:res} we construct the minimal resolution of $\A_h$ as an $\A_h$-bimodule. In particular, this allows us to give an explicit description of $\hoch^2(\A_h)$ in positive characteristic. The aim of Section \ref{S:hom} is to complete the construction of a contracting homotopy for the minimal resolution, and in Section \ref{S:Gerst} we recall the method developed by Su\'arez-\'Alvarez \cite{mSA17} to compute the brackets $[\hoch^1(A), \hoch^n(A)]$ for any associative unital algebra $A$. This allows us to obtain in Section \ref{S:Gerst:char0} the main results of this article: the description, in characteristic zero, of the Lie structure of $\hoch^\bullet(\A_h)$ as an $\hoch^1(\A_h)$-Lie module.

Below we summarize, in simplified from, the main results of the paper.

\begin{thmA}[Theorem~\ref{T:free:char:p}]
Assume  $\chara(\FF)=p>0$ and let $\mathsf{Z}(\A_h)$ denote the center of $\A_h$. Then $\hoch^2(\A_h)$ is a free $\mathsf{Z}(\A_h)$-module if and only if $\gcd(h, h')=1$. In this case, $\hoch^2(\A_h)$ has rank one over $\mathsf{Z}(\A_h)$ and, moreover, $\hoch^\bullet(\A_h)$ is a free $\mathsf{Z}(\A_h)$-module.
\end{thmA}

In positive characteristic, an explicit description of $\hoch^2(\A_h)$ is given in Theorem~\ref{T:res:HH2charp}, although this is a bit involved. On the other hand, in characteristic zero, $\hoch^2(\A_h)$ can be presented as a space of polynomials.

\begin{thmB}[Corollary~\ref{C:res:HH2inchar0}]
Assume  $\chara(\FF)=0$. There are isomorphisms
\begin{equation*}
\hoch^2(\A_h) \cong \A_h/\gcd(h, h')\A_h\cong \D[\y],
\end{equation*}
where $\D=\left(\FF[x]/\gcd(h, h')\FF[x]\right)$.
In particular, $\hoch^2(\A_h)=0$ if and only if $h$ is a separable polynomial; otherwise, $\hoch^2(\A_h)$ is infinite dimensional. 
\end{thmB}

The Hochschild cohomology $\hoch^{\bullet}(\A_h)=\bigoplus_{n\geq 0}\hoch^n(\A_h)$ can be made into a Lie module for the Lie algebra $\hoch^{1}(\A_h)$ of outer derivations of $\A_h$, under the \textit{Gerstenhaber bracket}. By the Hochschild-Kostant-Rosenberg Theorem, under suitable assumptions, this bracket is the generalization to higher degrees of the Schouten-Nijenhuis bracket. In our setting this is especially interesting in case $\chara(\FF)=0$ and $\gcd(h, h')\neq 1$ as then the description of $\hoch^{1}(\A_h)$ is related to the Witt algebra and, as we shall see, the $\hoch^{1}(\A_h)$-Lie module structure of $\hoch^{2}(\A_h)$ can be described in terms of the representation theory of the Witt algebra.

\begin{thmC}[\textit{cf.}\ Theorem~\ref{T:Gerst:char0:main}]
Assume that $\chara(\FF)=0$ and $\gcd(h, h')\neq 1$. Let $m_h+1$ be the largest exponent occurring in the decomposition of $h$ in $\FF[x]$ into irreducible factors. The structure of $\hoch^2(\A_h)$ as a Lie module, under the Gerstenhaber bracket, for the Lie algebra $\hoch^1(\A_h)$ is as follows:
\begin{enumerate}[label=\textup{(\alph*)}]
\item There is a filtration of length $m_h$ by $\hoch^1(\A_h)$-submodules
\begin{equation*}
\hoch^2(\A_h)=P_0\supsetneq P_1\supsetneq \cdots \supsetneq P_{m_h-1}\supsetneq P_{m_h}=0. 
\end{equation*} 
such that each factor $P_i/P_{i+1}$ is semisimple.
\item The irreducible summands of each $P_i/P_{i+1}$ can be naturally seen as obtained from intermediate series modules for the Witt algebra, under a suitable finite field extension of $\FF$.
\item $\hoch^2(\A_h)$ has finite composition length, equal to the number of irreducible factors of $\gcd(h, h')$, counted with multiplicity. 
\item $\hoch^2(\A_h)$ is a semisimple $\hoch^1(\A_h)$-module if and only if $h$ is not divisible by the cube of any non-constant polynomial.
\end{enumerate}
\end{thmC}

It is noteworthy that, in case $\FF$ is of characteristic $0$ and algebraically closed (so that each irreducible factor of $h$ is linear and the corresponding factor algebra of $\FF[x]$ is isomorphic to $\FF$), then from this theorem and previous results obtained in \cite{BLO15ja} we can recover the number of irreducible factors appearing in $h$ and the corresponding multiplicities. More specifically, let $\lambda(h)$ denote the partition encoding the multiplicities of the irreducible factors of $h$. We can conclude that if $\lambda(h)$ and $\lambda(g)$ are different partitions, then $\A_{h}$ is not derived equivalent to $\A_{g}$.

We now fix some definitions and notation. Given an associative algebra $A$ and elements $a,b \in A$, we use the commutator notation $[a,b]=ab-ba$. The center of $A$ and the centralizer of an element $a\in A$ will be denoted by $\mathsf{Z}(A)$ and $\mathsf{C}_{A}(a)$, respectively. An element $c\in A$ is \textit{normal} if $cA=Ac$ (an ideal of $A$). We remark that the set of normal elements of $A$ forms a multiplicative monoid.

Given a two-sided ideal $I$ of $A$, we will write $a\equiv b\modd{I}$ to mean that $a-b\in I$. This yields an equivalence relation on $A$ with the usual stability properties under addition and multiplication. If $J$ is another ideal such that $J\subseteq I$, then obviously $a\equiv b\modd{J}$ implies $a\equiv b\modd{I}$.  In case $I=cA$ for some normal element $c\in A$, we also use the notation $a\equiv b\modd{c}$.

Unadorned $\ot$ will always mean $\ot_\FF$. For any set $E$, $1_E$ will denote the identity map on $E$.
Given  $f\in\FF[x]$, 
$f^{(k)}$ stands for the $k$-th derivative of $f$ with respect to $x$, which we also denote by $f'$ and $f''$ in case $k=1, 2$, respectively. If $f, g\in\FF[x]$ are not both zero, then we tacitly assume that $\gcd(f, g)$ is monic.

An infinite-dimensional Lie algebra which plays an important role in the description of $\hoch^1(\A)$ is the \textit{Witt algebra}. A confusion with terminology may arise here, since the term Witt algebra has been used in the literature to mean two different things: the complex Witt algebra is the Lie algebra of derivations of the ring $\CC[z^{\pm 1}]$, with basis elements $w_n=z^{n+1}\frac{d}{dx}$, for $n\in\ZZ$;
while over a field $\KK$ of characteristic $p>0$, the Witt algebra is defined to be the Lie algebra of derivations of the ring
$\KK[z]/(z^p)$, spanned by $w_n$ for $-1\le n \le p-2$. 
Here we are considering a subalgebra of the first one (defined over the field $\FF$): 
\begin{equation}\label{E:def:Witt}
 \W = \spann_\FF\{w_i \mid i \geq -1\}, 
\end{equation}
equipped with the Lie bracket $\gb{w_m, w_n}=(n-m)w_{m+n}$, for $m, n\geq -1$. It is easy to check that if $\chara(\FF)=0$, then $\W$ is a simple Lie algebra (\textit{cf.}\ \cite[Lem.\ 5.19]{BLO15ja}). For the sake of simplicity and in accordance with the usage in \cite{BLO15ja}, we will abuse terminology and refer to the algebra $\W$ defined above as the Witt algebra. To make the distinction clear, we'll call the Lie algebra of derivations of $\FF[z^{\pm 1}]$, with basis $\{w_i\}_{i \in\ZZ}$, the \textit{full Witt algebra}.

A related Lie algebra of utmost importance in theoretical physics is the \textit{Virasoro algebra}, denoted by $\mathsf{Vir}$. It has basis $\{w_i \mid i \in\ZZ\}\cup\{ c \}$ over $\FF$, with bracket
\begin{equation*}
[c, \mathsf{Vir}]=0 \quad\mbox{and}\quad  \gb{w_m, w_n}=(n-m)w_{m+n}+\delta_{m+n, 0}\frac{m^3-m}{12}c,
\end{equation*}
for all $m, n\in\ZZ$. We will see in \eqref{E:Gerst:char0:struct:Umu:action} that the composition factors of $\hoch^2(\A_h)$ can be naturally embedded into irreducible modules for the Virasoro algebra. These are the so-called intermediate series modules and it is a result of Mathieu~\cite{oM92} that a Harish-Chandra module for $\mathsf{Vir}$ is either a highest weight module, a lowest weight module or an intermediate series module.

\medskip

\noindent  {\bf Acknowledgments:}  \  We thank Ken Brown for asking us questions motivating the topic of this paper. We would also like to express our gratitude to Quanshui Wu for kindly providing an argument confirming our conjecture on the description of the Nakayama automorphism of $\A_h$.

\section{Some technical results on commutators}\label{S:comm}

In this short section, we gather some technical lemmas about commutators in $\A_h$. We will need several additional results on centralizers and commutators in $\A_h$ from \cite{BLO15tams}, which for convenience we combine below.

\begin{prop}[{\textit{cf.}\ \cite[Lem.\ 3.4, 5.2, 6.1, 6.3; Prop.\ 5.5, 6.2; Thm.\ 5.3]{BLO15tams}}]\label{P:general:facts:Ah}
Let $\delta: \FF[x] \rightarrow \FF[x]$ be the derivation defined by $\delta(f) = f'h$ for all $f \in \FF[x]$.
\begin{enumerate}[label=\textup{(\alph*)}]
\item We have the following formula for computing in $\A_h$: 
\begin{equation}\label{eq:Ahcom}
\y^n f = \sum_{j=0}^n {n \choose j} \delta^j (f) \y^{n-j}.
\end{equation}
\item $\A_h$ is a free left $\FF[x]$-module with basis $\left\{h^i y^i\right\}_{i\geq 0}$.
\item If $\chara(\FF)=0$, then $\cent{h}=\FF$; if $\chara(\FF)=p>0$, then $\cent{h}$ is the polynomial algebra in the variables $x^p$ and $h^p y^p$.
\item The centralizer $\mathsf{C}_{\A_h}(x)$ is generated by $\FF[x]$ and $\cent{h}$.
\item $\A_h$ is free over $\cent{h}$ and over $\mathsf{C}_{\A_h}(x)$. If $\chara(\FF)=p>0$, then 
\begin{equation*}
\A_h=\bigoplus_{i, j=0}^{p-1}\cent{h} x^ih^jy^j =\bigoplus_{j=0}^{p-1}\mathsf{C}_{\A_h}(x) h^jy^j.
\end{equation*}
\item $[\A_h, \A_h]\subseteq h\A_h$. If $\chara(\FF)=0$, then $[x, \A_h]=[\y, \A_h]=[\A_h, \A_h]=h\A_h$.
\end{enumerate}
\end{prop}

\begin{lemma}\label{L:comm:bx}
For any $0\neq h\in\FF[x]$, $[\FF[x], \A_h]=[x, \A_h]$.
\end{lemma}

\begin{proof}
If $\chara(\FF)=0$, then the claim follows from $[x, \A_h]=[\A_h, \A_h]$, by Proposition~\ref{P:general:facts:Ah}. 

So assume $\chara(\FF)=p>0$. By \cite[Lem.\ 6.3]{BLO15tams} and Proposition~\ref{P:general:facts:Ah}, we know that 
\[ [x, \A_{h}] = \bigoplus_{j=0}^{p-2}h \mathsf{C}_{\A_h}(x) h^j y^j \hbox{ and } \A_h = \bigoplus_{j=0}^{p-1} \mathsf{C}_{\A_h}(x) h^j y^j. \]
Given $f\in\FF[x]$, $c\in\mathsf{C}_{\A_h}(x)$ and $0\leq j\leq p-1$ we have, using \eqref{eq:Ahcom}:
\begin{align*}
[ch^jy^j, f]= ch^j[y^j, f]=ch^j\sum_{k=1}^j {j\choose k} f^{(k)}y^{j-k}=h\sum_{k=1}^j {j\choose k}ch^{k-1} f^{(k)}h^{j-k}y^{j-k}.
\end{align*}
So, $[ch^jy^j, f] \in  \bigoplus_{j=0}^{p-2}h \mathsf{C}_{\A_h}(x) h^j y^j=[x, \A_{h}].$
\end{proof}

Now we can characterize the subspace $[x, \A_h]+[\y, \A_h]$ in case $\chara(\FF)=p>0$.

\begin{lemma}\label{L:comm:adxpy}
Assume  $\chara(\FF)=p>0$. The following hold:
\begin{enumerate}[label=\textup{(\alph*)}]
\item for all $z\in\cent{h}$, $f\in\FF[x]$ and $0\leq j\leq p-2$, we have
\begin{equation*}
[\y, zfh^jy^j]\in [x, \A_h] \quad \mbox{and} \quad [\y, zfh^{p-1}y^{p-1}]=zhf' h^{p-1}y^{p-1};
\end{equation*}%
\label{L:comm:adxpy:1}
\item $\displaystyle [x, \A_h]+[\y, \A_h]=\bigoplus_{\substack{i, j=0\\ (i, j)\neq (p-1, p-1)}}^{p-1}\cent{h} hx^ih^jy^j$;
\label{L:comm:adxpy:2}
\item $\displaystyle h\A_h=\left( [x, \A_h]+[\y, \A_h]\right)\oplus h\cent{h} x^{p-1}h^{p-1}y^{p-1}$.\label{L:comm:adxpy:3}
\end{enumerate}
\end{lemma}

\begin{proof}
For the first part of (a), if suffices to show that  $[\y, fh^jy^j]\in [x, \A_h]$ for all $0\leq j\leq p-2$, as the latter is clearly a $\cent{h}$-module. Since $\y-hy=h'\in\FF[x ]$ and $[\FF[x], \A_h]=[x, \A_h]$, we need to prove that $[hy, fh^jy^j]\in [x, \A_h]$. Moreover, 
\begin{equation*}
[hy, fh^jy^j]=[hy, f]h^jy^j+f[hy, h^jy^j]=hf'h^jy^j+f[hy, h^jy^j]
\end{equation*}
and $hf'h^jy^j\in[x, \A_h]$, so we are left with showing that $[hy, h^jy^j]\in[x, \A_h]$. This is clear for $j=0, 1$, and for $2\leq j\leq p-2$ we have, using \eqref{eq:Ahcom}:
\begin{align*}
 [hy, h^jy^j]&=[hy, h^j]y^j + h^j[hy, y^j]=jh'h^jy^j + h^j[h, y^j]y\\ &=jh'h^jy^j - h^j\sum_{k=1}^j {j\choose k} h^{(k)}y^{j-k+1}=- h^j\sum_{k=2}^j {j\choose k} h^{(k)}y^{j-k+1}\\ &=-\sum_{\ell=1}^{j-1} {j\choose \ell-1} h^{j-\ell-1}h^{(j-\ell+1)}h^{\ell+1}y^{\ell}.
\end{align*}
This proves that $[\y, zfh^jy^j]\in [x, \A_h]$ for all $z\in\cent{h}$, $f\in\FF[x]$ and $0\leq j\leq p-2$.

Now notice that, since $h^p, y^p\in\cent{1}$, then 
\begin{equation}\label{E:comm:adxpy:yhatcomm}
h^{p-1}y^{p-1}\y=h^{p-1}y^{p}h=h^p y^p=yh^{p}y^{p-1}=\y h^{p-1}y^{p-1},
\end{equation}
so $[\y, h^{p-1}y^{p-1}]=0$. Thus, for $z\in\cent{h}$ and $f\in\FF[x]$ we have
\begin{equation*}
[\y, zfh^{p-1}y^{p-1}]=z[\y, f]h^{p-1}y^{p-1}=zhf' h^{p-1}y^{p-1},
\end{equation*}
which finishes the proof of (a).

Since $\displaystyle \cent{h} h\cdot \im\left( \frac{d}{dx}\right)h^{p-1}y^{p-1}=\bigoplus_{i=0}^{p-2}\cent{h} hx^ih^{p-1}y^{p-1}$, (b) is also established and (c) follows from (b), by Proposition~\ref{P:general:facts:Ah}.
\end{proof}

\section{Minimal free bimodule resolution of $\A_h$}\label{S:res}

For simplicity, throughout the remainder of this paper, we denote $\A_h$ simply by $\A$, reserving the notation $\A_h$ for situations in which we want to emphasize $h$ or make particular choices for $h$, e.g.\ when referring to the Weyl algebra $\A_1$.

In this section, we construct a free resolution of $\A$ as an $\A$-bimodule or, equivalently, as a left $\A^{e}$-module, where $\A^{e}=\A\otimes\A^{op}$ is the enveloping algebra of $\A$ and $\A^{op}$ is the opposite algebra of $\A$. 

We will follow the approach in \cite{CS15}. Let $\V=\FF x\oplus\FF\y$ be the vector subspace of $\A$ spanned by $x$ and $\y$ and let $\R=\FF\rr$ be a vector space of dimension $1$. Consider the following sequence of right $\A$-module maps:
\begin{equation}\label{E:res:K}
\begin{tikzcd}
0\arrow[r] &[-.5em] \A\ot\R\ot\A \arrow[r, "\dd_1"] & \A\ot\V\ot\A \arrow[r, "\dd_0"] \arrow[l, dashed, bend left=20, "\s_1"]  & \A\ot\A \arrow[r, "\mu"] \arrow[l, dashed, bend left=20, "\s_0"] &  \arrow[r] \A \arrow[l, dashed, bend left=30, "\s_{-1}"]  & [-.5em] 0.
\end{tikzcd}
\end{equation}
The maps $\mu, \dd_0$ and $\dd_1$ are in fact $\A$-bimodule maps, whereas the maps $\s_{-1}, \s_0$ and $\s_1$ are just right $\A$-module maps. We describe them all below, except for $\s_1$, which we discuss only in Section~\ref{S:hom}:
\begin{itemize}
\item $\mu$ is the multiplication map;
\item $\dd_0(1\ot v\ot 1)=v\ot 1-1\ot v$ for all $v\in\V$;
\item $\s_{-1}(1)=1\ot 1$;
\item $\s_0(x^k\y^\ell\ot 1)=\sum_{i=0}^{k-1} x^i\ot x\ot x^{k-1-i}\y^\ell + \sum_{j=0}^{\ell-1} x^k\y^j\ot \y\ot \y^{\ell-1-j}$, with the usual convention that an empty summation is null; in particular, $\s_0(1\ot 1)=0$;
\item $\dd_1(1\ot \rr \ot 1)=1\ot\y\ot x + \y\ot x\ot 1 -1\ot x\ot \y - x\ot\y\ot 1-\s_0(h\ot 1)$.
\end{itemize}

It is easy to check that 
\begin{equation}\label{E:res:chain}
\mu\circ\dd_0=0=\dd_0\circ\dd_1,
\end{equation}
so \eqref{E:res:K} is a complex of $\A$-bimodules. In fact, we already know that \eqref{E:res:K} is exact, and hence a free resolution of $\A$, since its associated graded complex is exact (see \cite{CS15}), but it will be useful for further computations to have an explicit contracting homotopy. 

We claim that the right $\A$-module maps $\s_{-1}, \s_0$ and $\s_1$ form the desired contracting homotopy for \eqref{E:res:K}, i.e., that the following hold:
\begin{equation}\label{E:res:homotopies}
\begin{split}
\mu\circ \s_{-1} &=1_{\A},\\
\s_{-1}\circ\mu + \dd_0\circ \s_0 &=1_{\A\ot\A},\\
\s_{0}\circ\dd_0 + \dd_1\circ \s_1 &=1_{\A\ot\V\ot\A},\\
\s_1\circ \dd_1 &=1_{\A\ot\R\ot\A}.
\end{split}
\end{equation}
The first two equalities are easy to prove and are left as an exercise. So as not to stray from the main ideas of this section, we will defer the construction of the map $\s_1$ and the proof of the last two relations in \eqref{E:res:homotopies} until Section~\ref{S:hom} (see Theorem~\ref{T:hom:main}).

Applying the functor $\Hom_{\A^e}(-, \A)$ to the 
resolution associated with \eqref{E:res:K}, we get the commutative diagram

\begin{equation*}
\begin{tikzcd}[row sep=5ex, column sep = .9em]
0\arrow[r] &  [-.5em] \Hom_{\A^e}(\A\ot\A, \A) \arrow[r, "\dd^*_0"] \arrow[d, "\rho_0"] & \Hom_{\A^e}(\A\ot\V\ot\A, \A) \arrow[r, "\dd^*_1"] \arrow[d, "\rho_1"]  & \Hom_{\A^e}(\A\ot\R\ot\A, \A) \arrow[r] \arrow[d, "\rho_2"]  & [-.5em] 0\\
0\arrow[r]  &\A \arrow[r, "\rD_1"] & \A\oplus\A \arrow[r, "\rD_2"] & \A \arrow[r] & 0.
\end{tikzcd}
\end{equation*}

\medskip

\noindent
where $\dd^*_i$ is right composition with $\dd_i$, for $i=0, 1$, and the vector space isomorphisms $\rho_j$ are defined as usual by:
\begin{equation*}
\rho_0(f)=f(1\ot 1), \quad \rho_1(f)=(f(1\ot x\ot 1), f(1\ot\y\ot 1)), \quad \rho_2(f)=f(1\ot\rr \ot 1). 
\end{equation*}

The maps $\rD_1$ and $\rD_2$ are given by:
\begin{align}\label{E:res:defD1}
&\rD_1(\alpha) =([x, \alpha], [\y, \alpha]) \quad \mbox{and}\\
&\rD_2(\alpha, \beta) = [\beta, x] + [\y, \alpha] -F_\alpha (h),\label{E:res:defD2}
\end{align}
for all $\alpha, \beta \in\A$, where $F_\alpha:\FF[x]\longrightarrow\A$ is the linear map defined by 
\begin{equation}\label{E:res:defFalpha}
F_\alpha(x^s)=\sum_{\ell=0}^{s-1}x^\ell \alpha x^{s-\ell-1}, \quad \mbox{for $s\geq 0$,}
\end{equation}
with the convention that $F_\alpha(1)=0$.

Since $F_{z\alpha}=zF_\alpha$, for $z\in\mathsf{Z}(\A)$, the maps $\rho_i$ and $\rD_j$ are actually $\mathsf{Z}(\A)$-module maps. It follows that, as a $\mathsf{Z}(\A)$-module, the Hochschild cohomology of $\A$ can be determined from the maps $\rD_i$:
\begin{itemize}
\item $\hoch^0(\A)=\mathsf{Z}(\A)=\ker \rD_1$;
\item $\hoch^1(\A)= \der (\A) /\inder (\A) \cong \ker \rD_2/\im \rD_1$;
\item $\hoch^2(\A) \cong \A/\im \rD_2$ is the space of equivalence classes of infinitesimal deformations of $\A$ (see \cite{mG64});
\item $\hoch^i(\A) =0$ for all $i\geq 3$.
\end{itemize}
The degree zero cohomology $\hoch^0(\A)$ has been computed in \cite[Section 5]{BLO15tams}, while the derivations and the Lie algebra structure of $\hoch^1(\A)$ were determined in \cite{BLO15ja}, both over arbitrary fields. 

\begin{exams} 
Assume $\chara(\FF)=0$. 
\begin{itemize}
\item If $h=1$, then $\A_1$ is the Weyl algebra and it is well known (see \cite{rS61}) that $\hoch^0(\A_1)=\FF$ and $\hoch^i(\A_1)=0$ for all $i>0$. In this case, $\A_1$ is graded, setting $\deg(x)=1$ and $\deg(y)=-1$.
\item If $h=x$, then $\A_x$ is the universal enveloping algebra of the two-dimensional non-abelian Lie algebra. In this case, $\hoch^0(\A_x)=\FF=\hoch^1(\A_x)$, by \cite[Thm.\ 5.29]{BLO15ja}. We will see shortly that $\hoch^2(\A_x)=0$.
\item If $h=x^2$, then $\A_{x^2}$ is the Jordan plane. In this case, $\A_{x^2}$ is graded, setting 
$\deg(x)=\deg(\y)=1$. Note that $\hoch^0(\A_{x^2})=\FF$ and by \cite[Thm.\ 5.29]{BLO15ja}, as a Lie algebra, $\hoch^1(\A_{x^2})=\FF c\oplus \W$, where $c$ is central and $\W$ is  the Witt algebra given in \eqref{E:def:Witt}.
We will see that $\hoch^2(\A_{x^2})\cong \FF[\y]$ is naturally a simple module for $\W$ and that this module can be embedded into a simple module for the Virasoro algebra.
\end{itemize}
\end{exams}

\medskip

Our main goal in this section will be to determine the image of $\rD_2$ and the quotient $\cent{}$-module $\A/\im \rD_2$. Later we will determine the Lie action of $\hoch^1(\A)$ on $\hoch^2(\A)$ given by the Gerstenhaber bracket. Towards that goal, we start out by studying the map $F_\alpha$ given in \eqref{E:res:defFalpha}. It will be convenient to introduce a mild generalization, so that $F_\alpha$ can be defined for all $\alpha$ in the Weyl algebra $\A_1\supseteq \A$. With this extension, the range of $F_\alpha$ becomes $\A_1$, but we will still use $F_\alpha$ to denote this map.

\begin{lemma}\label{L:res:propsFalpha}
For $\alpha\in\A_1$, let $F_\alpha:\FF[x]\longrightarrow\A_1$ be the linear map defined by \eqref{E:res:defFalpha}. The following hold for all $f, g\in\FF[x]$:
\begin{enumerate}[label=\textup{(\alph*)}]
\item $F_\alpha(fg)=fF_\alpha(g) + F_\alpha(f)g$, i.e., $F_\alpha$ is a derivation.\label{L:res:propsFalpha:1}
\item If $\alpha\in\mathsf{C}_{\A_1}(x)$ then $F_\alpha(f)=\alpha f'$.\label{L:res:propsFalpha:3}
\item Moreover, if $\alpha\in\A$, then $F_\alpha(f)\in f'\alpha + [x, \A]$.\label{L:res:propsFalpha:4}
\end{enumerate}
\end{lemma}

\begin{proof}
To show (a), it suffices to consider $f=x^k$ and $g=x^s$, with $k, s\geq 0$. Then:
\begin{align*}
 F_\alpha(fg) &= F_\alpha(x^{k+s})=\sum_{\ell=0}^{k+s-1}x^{k+s-\ell-1} \alpha x^{\ell}\\
 &=x^k\sum_{\ell=0}^{s-1}x^{s-\ell-1} \alpha x^{\ell} + \left(\sum_{\ell=0}^{k-1}x^{k-\ell-1} \alpha x^{\ell}\right)x^s\\
 &=fF_\alpha(g) + F_\alpha(f)g.
\end{align*}
This proves (a); (b) is clear and we proceed to prove (c). Again, we need only consider $\alpha\in\A$ and $f=x^k$, as above. We have:
\begin{align*}
F_\alpha(x^k) &= \sum_{\ell=0}^{k-1}x^{k-\ell-1} \alpha x^\ell= \sum_{\ell=0}^{k-1}x^{k-1} \alpha + \sum_{\ell=0}^{k-1}x^{k-\ell-1} [\alpha, x^\ell]\\
&=kx^{k-1} \alpha + \sum_{\ell=0}^{k-1}[x^{k-\ell-1} \alpha, x^\ell] \\ &\in kx^{k-1} \alpha + [\FF[x], \A] =f' \alpha + [x, \A].
\end{align*}
\end{proof}

In case $\chara(\FF)=0$, the following result completely describes the image of the map $\rD_2$. 

\begin{prop}\label{P:res:imd2char0}
The following hold:
\begin{enumerate}[label=\textup{(\alph*)}]
\item $\displaystyle \im \rD_2\subseteq \gcd(h, h')\A$.\label{P:res:imd2char0:1}
\item If $\chara(\FF)=0$ then $\displaystyle \im \rD_2= \gcd(h, h')\A$.\label{P:res:imd2char0:2}
\end{enumerate}
\end{prop}

\begin{proof}
It is convenient to write $\rD_2= \rD_2^1 \oplus \rD_2^2$, where 
\begin{equation}\label{E:res:splitD2}
\begin{array}{rccl}
 \rD_2^1\ : & \A  & \longrightarrow &\A \\
   &  \alpha & \mapsto  & [\y, \alpha] -F_\alpha(h)
\end{array}
\quad \mbox{and} \quad
\begin{array}{rccl}
 \rD_2^2\ : & \A  & \longrightarrow &\A \\
   &  \beta & \mapsto  & [\beta, x]
\end{array}
.
\end{equation}

Since, by Lemma~\ref{L:res:propsFalpha}\,(c), 
\begin{align*}
\rD_2^1(-\alpha) \in h'\alpha + [x, \A] + [\y, \A]
 \subseteq h'\A + h\A=\gcd(h, h')\A,
\end{align*}
for all $\alpha\in\A$, it follows that
\begin{align*}
\im \rD_2 &=\im \rD_2^1 + \im \rD_2^2
\subseteq \gcd(h, h')\A + [x, \A]\\
&\subseteq \gcd(h, h')\A + h\A
= \gcd(h, h')\A.
\end{align*}

Now assume $\chara(\FF)=0$. By Proposition~\ref{P:general:facts:Ah}, we know that $[x, \A]=[\y, \A]=h\A$ and thus $\im \rD_2^2=[x, \A]=h\A$, which implies that $h\A\subseteq\im \rD_2$. Hence, we proceed to show that also $h'\A\subseteq\im \rD_2$. For $\alpha\in\A$, we have seen that 
\begin{equation*}
\rD_2^1(-\alpha)-h'\alpha\in[\alpha, \y]+ [x, \A]\subseteq h\A\subseteq \im \rD_2.
\end{equation*}
Also, $\rD_2^1(-\alpha)\in\im \rD_2$, so it follows that $h'\alpha\in\im \rD_2$. Hence, $\gcd(h, h')\A=h'\A + h\A\subseteq \im \rD_2$ and the equality holds, by (a).
\end{proof}

\begin{cor}\label{C:res:HH2inchar0}
Assume  $\chara(\FF)=0$. There are isomorphisms 
\begin{equation*}
\hoch^2(\A) \cong \A/\gcd(h, h')\A\cong \D[\y],
\end{equation*}
where $\D=\left(\FF[x]/\gcd(h, h')\FF[x]\right)$.
In particular, $\hoch^2(\A)=0$ if and only if $\gcd(h, h')=1$, i.e., if and only if $h$ is a separable polynomial; otherwise, $\hoch^2(\A)$ is infinite dimensional.
\end{cor}

Let us now consider the case $\chara(\FF)=p>0$. Suppose first that $h\in\FF[x^p]$, a central polynomial. This is a particularly interesting case, not only because it includes the Weyl algebra $\A_1$, but also since $\A_h$ is Calabi-Yau if and only if $h$ is central. Indeed, more generally, $\A_h$ is twisted Calabi-Yau with Nakayama automorphism satisfying $x\mapsto x$, $\y\mapsto\y+h'$, a fact which can be derived from \cite[Rmk.\ 3.4, (2.10)]{LLW14}.

Although we can retrieve the following result from Theorem~\ref{T:res:HH2charp} below, we think this particular case helps set the stage for our general result and offers a more concrete example.

\begin{prop}\label{P:res:HH2hcentral}
Assume  $\chara(\FF)=p>0$ and $0\neq h\in\FF[x^p]$. Then $\im \rD_2=[x, \A]+[\y, \A]$. Thus:
\begin{align*}
 \hoch^2(\A) \cong  \bigoplus_{\substack{i, j=0\\ (i, j)\neq (p-1, p-1)}}^{p-1} \left(\cent{h}/ h\cent{h}\right) x^ih^jy^j\oplus \mathsf{Z}(\A) x^{p-1}h^{p-1}y^{p-1},
\end{align*}
as $\mathsf{Z}(\A)$-modules.

In particular, in case $h=1$ we obtain $\hoch^2(\A_1) \cong \mathsf{Z}(\A_1) x^{p-1}y^{p-1}$, a rank-one module over $\mathsf{Z}(\A_1)=\FF[x^p, y^p]$.
\end{prop}

\begin{proof}
We continue to use the maps $\rD_2^1$ and $\rD_2^2$ defined in \eqref{E:res:splitD2}.
For $\alpha\in\A$ we have
\begin{equation}\label{E:res:HH2hcentral}
\rD_2^1(\alpha)=[\y, \alpha] -F_\alpha(h)=[\y, \alpha] -h'\alpha -\Theta_\alpha=[\y, \alpha]  -\Theta_\alpha,
\end{equation}
for some $\Theta_\alpha\in[x, \A]=\im \rD_2^2$. Thus, $\im \rD_2^1\subseteq [x, \A]+[\y, \A]$ and there are inclusions $[x, \A]\subseteq \im \rD_2=\im \rD_2^1 + \im \rD_2^2\subseteq [x, \A]+[\y, \A]$. Conversely, by \eqref{E:res:HH2hcentral} we also have that $[\y, \alpha]=\rD_2^1(\alpha)+\Theta_\alpha\in\im \rD_2^1 + \im \rD_2^2=\im \rD_2$, so $[\y, \A]\subseteq\im \rD_2$, yielding the equality $\im \rD_2=[x, \A]+[\y, \A]$. 

The expression for $\A/\im \rD_2$ then comes from Lemma~\ref{L:comm:adxpy}\,(b) and Proposition~\ref{P:general:facts:Ah}.
\end{proof}

We now tackle the general case for $0\neq h\in\FF[x]$, which is a bit more intricate than the particular case studied above. Consider the decomposition $\A=\mathcal{I}\oplus\mathcal{J}$, where 
\begin{equation}\label{E:res:def:I:J}
\mathcal{I}=\mathsf{C}_{\A}(x)h^{p-1}y^{p-1} \quad \mbox{and} \quad \mathcal{J}=\bigoplus_{j=0}^{p-2} \mathsf{C}_{\A}(x)h^jy^j.
\end{equation}
Thus, $\im \rD_2^1=\im \rD_2^1|_{\mathcal{I}} + \im \rD_2^1|_{\mathcal{J}}$. Also, by \cite[Lem.\ 6.3\,(b)]{BLO15tams}, $\im \rD_2^2=[x, \A]=h\mathcal{J}$.

We wish to show that 
\begin{equation}\label{E:res:imrestJ}
\im \rD_2^1|_{\mathcal{J}} + \im \rD_2^2=h\mathcal{J} + h'\mathcal{J}=\gcd(h, h')\mathcal{J}.
\end{equation}
Let $\alpha\in\mathcal{J}$. Then $[\y, \alpha]\in[x, \A]=h\mathcal{J}$, by Lemma~\ref{L:comm:adxpy}\,(a). As in \eqref{E:res:HH2hcentral}, $\rD_2^1(\alpha)=[\y, \alpha] -h'\alpha -\Theta_\alpha$ for some $\Theta_\alpha\in[x, \A]=h\mathcal{J}$. Thus, $\im \rD_2^1|_{\mathcal{J}}\subseteq h\mathcal{J} + h'\mathcal{J}$; moreover, $h'\alpha=-\rD_2^1(\alpha)+[\y, \alpha] -\Theta_\alpha  \in\im \rD_2^1|_{\mathcal{J}}+\im \rD_2^2$, and \eqref{E:res:imrestJ} is established.

So it remains to determine the image of $\rD_2^1|_{\mathcal{I}}$. Let $\alpha\in\mathcal{I}$. Without loss of generality, we can assume that $\alpha=zfh^{p-1}y^{p-1}$ with $z\in\mathsf{Z}(\A)$ and $f\in\FF[x]$. Then, using Lemma~\ref{L:comm:adxpy}\,(a), we have 
\begin{equation}\label{E:res:computD21I}
\begin{split}
\rD_2^1(\alpha)&=[\y, zfh^{p-1}y^{p-1}] -F_\alpha(h)\\
&=zf'hh^{p-1}y^{p-1}-zh'fh^{p-1}y^{p-1}-\Theta_\alpha\\
&=z(f'h-h'f)h^{p-1}y^{p-1}-\Theta_\alpha,
\end{split}
\end{equation}
with $\Theta_\alpha\in[x, \A]=h\mathcal{J}$. 

Define the map 
\begin{equation}\label{E:res:defkappa}
\varkappa=\varkappa_h:\FF[x]\longrightarrow\FF[x], \quad \varkappa(g)=g'h-h'g.
\end{equation}
By \cite[Lem.\ 4.28\,(d)]{BLO15ja}, we know that $\ker\varkappa=\FF[x^p]\left(h/\varrho_h\right)$, where $\varrho_h$ is the unique monic polynomial in $\FF[x^p]$ of maximal degree dividing $h$ (see \cite[Definition 2.14]{BLO15ja} for a detailed description of $\varrho_h$). Since $\varkappa$ is clearly $\FF[x^p]$-linear and $\FF[x]$ is free of rank $p$ over the hereditary algebra $\FF[x^p]$, we conclude that $\mathcal{K}:=\im \varkappa$ is a free $\FF[x^p]$-submodule of $\FF[x]$ of rank $p-1$. 

From the above and \eqref{E:res:computD21I} we can conclude that $\im \rD_2^1|_{\mathcal{I}}+\im \rD_2^2=h\mathcal{J}\oplus\mathsf{Z}(\A)\mathcal{K}h^{p-1}y^{p-1}$ and finally that
\begin{equation}\label{E:res:imD2charp}
\im \rD_2=\gcd(h, h')\mathcal{J}\oplus\mathsf{Z}(\A)\mathcal{K}h^{p-1}y^{p-1}.
\end{equation}
Thence, we obtain a description of $\hoch^2(\A)$ in positive characteristic.

\begin{thm}\label{T:res:HH2charp}
Assume  $\chara(\FF)=p>0$. Then the image of the map $\rD_2$ defined in \eqref{E:res:defD2} is $\im \rD_2=\gcd(h, h')\mathcal{J}\oplus\mathsf{Z}(\A)\mathcal{K}h^{p-1}y^{p-1}$, where 
$\mathcal{J}$ and $\varkappa$ are given in \eqref{E:res:def:I:J} and \eqref{E:res:defkappa}, respectively, and $\mathcal{K}$ is the image of $\varkappa$. Thus:
\begin{align*}
 \hoch^2(\A) \cong \mathcal{J}/\gcd(h, h')\mathcal{J}  \oplus \left(\mathsf{C}_{\A}(x)/\mathsf{Z}(\A)\mathcal{K}\right) h^{p-1}y^{p-1},
\end{align*}
as $\mathsf{Z}(\A)$-modules. In particular, $\hoch^2(\A)$ is nonzero for all $0\neq h\in\FF[x]$.
\end{thm}

\begin{remark}
Suppose that in Theorem~\ref{T:res:HH2charp} we take $0\neq h\in\FF[x^p]$. Then $\gcd(h, h')=h$ and $\mathcal{K}=h\, \im\frac{d}{dx}=\bigoplus_{i=0}^{p-2}\FF[x^p]hx^i$, so that 
\begin{align*}
\im \rD_2 &= h\mathcal{J} \oplus \bigoplus_{i=0}^{p-2}\mathsf{Z}(\A)hx^i h^{p-1}y^{p-1}=[x, \A]+[\y, \A],
\end{align*}
by Lemma~\ref{L:comm:adxpy}\,(b), in agreement with the statements in Proposition~\ref{P:res:HH2hcentral}. 
\end{remark}

\begin{exams} Let $\chara(\FF)=p>0$.
\begin{enumerate}[label=\textup{(\alph*)}]
\item In case $h=1$, then $\A_1$ is the Weyl algebra and, as observed in Proposition~\ref{P:res:HH2hcentral}, $\hoch^2(\A_1) \cong \mathsf{Z}(\A_1) x^{p-1}y^{p-1}$ is a rank-one free module over $\mathsf{Z}(\A_1)=\FF[x^p, y^p]$. It was shown in \cite[Thm.\ 3.8]{BLO15ja} that $\hoch^1(\A_1)$ is a rank-two free module over $\mathsf{Z}(\A_1)$.

\item In case $h=x$, then $\A_x$ is the universal enveloping algebra of the two-dimensional non-abelian Lie algebra. We have $\gcd(h, h')=1$ so that $\mathcal{J}/\gcd(h, h')\mathcal{J}=0$. By computing the image under $\varkappa$ of the $\FF[x^p]$-basis $\{x^i\mid 0\leq i\leq p-1\}$ of $\FF[x]$ we easily see that 
\begin{equation*}
\mathsf{Z}(\A_x)\mathcal{K}=\mathsf{Z}(\A_x)\oplus\bigoplus_{i=2}^{p-1}\mathsf{Z}(\A_x)x^i.
\end{equation*}
Hence, Theorem~\ref{T:res:HH2charp} yields 
\begin{equation*}
\hoch^2(\A_x) \cong  \mathsf{Z}(\A_x)x^{p}y^{p-1},
\end{equation*}
again a free rank-one module over $\mathsf{Z}(\A_x)=\FF[x^p, x^py^p]$.

\item Assume $h=x^2$. Then $\A_{x^2}$ is the Jordan plane. We distinguish between two cases:
\begin{itemize}
\item {\bf Case 1: $p=2$}. \\ In this case $x^2$ is central and we use Proposition~\ref{P:res:HH2hcentral} to obtain the isomorphism
\begin{equation*}
\hoch^2(\A_{x^2}) \cong \D \oplus \D x\oplus \D x^2y\oplus \mathsf{Z}(\A_{x^2})x^3y,
\end{equation*}
where $\mathsf{Z}(\A_{x^2})=\FF[x^2, x^4y^2]$ and $\D=\mathsf{Z}(\A_{x^2})/x^2\mathsf{Z}(\A_{x^2})$.

\item {\bf Case 2: $p>2$}. \\ In this case $x^2$ is not central and we use Theorem~\ref{T:res:HH2charp}. Since $\gcd(h, h')=x$ and $\mathsf{C}_{\A_{x^2}}(x)/x\mathsf{C}_{\A_{x^2}}(x)\cong \mathsf{Z}(\A_{x^2})/x^p\mathsf{Z}(\A_{x^2})$, we can conclude that $\mathcal{J}/\gcd(h, h')\mathcal{J}\cong\bigoplus_{j=0}^{p-2} \left(\mathsf{Z}(\A_{x^2})/x^p\mathsf{Z}(\A_{x^2})\right)h^jy^j$. Finally, as in the case $h=x$, it is easy to see that 
\begin{equation*}
\mathsf{Z}(\A_{x^2})\mathcal{K}=\bigoplus_{i=1}^{2}\mathsf{Z}(\A_{x^2})x^i\oplus\bigoplus_{i=4}^{p}\mathsf{Z}(\A_{x^2})x^i,
\end{equation*}
where the last summand is zero in case $p=3$. Hence Theorem~\ref{T:res:HH2charp} gives 
\begin{equation*}
\hoch^2(\A_{x^2}) \cong \D \oplus \D x^{2}y \oplus \mathsf{Z}(\A_{x^2})x^{4}y^{2},\quad \mbox{in case $p=3$, and}
\end{equation*}
\begin{align*}
\hoch^2(\A_{x^2}) &\cong  \bigoplus_{j=0}^{p-2} \D x^{2j}y^j \oplus \D x^{2(p-1)}y^{p-1}\oplus \mathsf{Z}(\A_{x^2})x^{2p+1}y^{p-1}\\
&=\bigoplus_{j=0}^{p-1} \D x^{2j}y^j \oplus \mathsf{Z}(\A_{x^2})x^{2p+1}y^{p-1},
\end{align*}
for all primes $p>3$, where $\mathsf{Z}(\A_{x^2})=\FF[x^p, x^{2p}y^p]$ and $\D=\mathsf{Z}(\A_{x^2})/x^p\mathsf{Z}(\A_{x^2})$.
\end{itemize}
Notice that in all cases, $\hoch^2(\A_{x^2})$ is not a free module over $\mathsf{Z}(\A_{x^2})$, although it is composed of a torsion summand and a free summand of rank one.
\end{enumerate}
\end{exams}

We have seen in the examples that, in general, $\hoch^2(\A)$ is not a free module over $\cent{}$. The next theorem provides a necessary and sufficient condition for $\hoch^2(\A)$ to be free.

\begin{thm}\label{T:free:char:p}
Assume  $\chara(\FF)=p>0$. Then $\hoch^2(\A)$ is a free $\mathsf{Z}(\A)$-module if and only if $\gcd(h, h')=1$. In this case, $\hoch^2(\A)$ has rank one over $\mathsf{Z}(\A)$ and, moreover, $\hoch^\bullet(\A)$ is a free $\mathsf{Z}(\A)$-module.
\end{thm}
\begin{proof}
The last statement follows from the first by \cite[Thm.\ 6.29]{BLO15ja}, so we need only focus on $\hoch^2(\A)$.

The condition $\gcd(h, h')=1$ is necessary, as otherwise  $\mathcal{J}/\gcd(h, h')\mathcal{J}$ would be nonzero and annihilated by the central element $(\gcd(h, h'))^p$. Next we prove that it is sufficient.

Suppose $\gcd(h, h')=1$. Then $\hoch^2(\A) \cong \left(\mathsf{C}_{\A}(x)/\mathsf{Z}(\A)\mathcal{K}\right) h^{p-1}y^{p-1}$ and, since $\mathsf{C}_{\A}(x)=\mathsf{Z}(\A)\FF[x]$, it suffices to prove that $\mathcal{K}$ is a direct summand of $\FF[x]$, as $\FF[x^p]$-modules. The latter is equivalent to showing that $\FF[x]/\mathcal{K}$ is torsion free, for then the canonical epimorphism $\FF[x]\rightarrow \FF[x]/\mathcal{K}$ will yield the decomposition $\FF[x]=\mathcal{K} \oplus\FF[x^{p}]\xi$, for some rank-one free $\FF[x^{p}]$-submodule $\FF[x^{p}]\xi$. It will follow that $\hoch^2(\A) \cong \mathsf{Z}(\A)\xi h^{p-1}y^{p-1}$, a free $\mathsf{Z}(\A)$-module of rank one. \medskip 

\textit{Claim:} {\it The $\FF[x^{p}]$-module $\FF[x]/\mathcal{K}$ is torsion free.}

\smallskip
\textit{Proof of the claim:}  
Let $0\neq\omega\in\FF[x^p]$ and $f\in\FF[x]$ be such that $\omega f\in\mathcal{K}$, say $\omega f=\varkappa(g)$. It needs to be shown that $f\in \mathcal{K}$. For such, it is enough to show that there exist $q\in\FF[x]$ and $r\in\FF[x^p]$ so that $g=\omega q+rh$. Indeed, if this is the case then 
$\omega f=\omega\varkappa(q)+r\varkappa(h)=\omega\varkappa(q)$ and it follows that $f=\varkappa(q)\in\mathcal{K}$.

\medskip 

\textit{Subclaim 1:} {\it $g\in\omega \FF[x] + h\FF[x]$.}

\smallskip
\textit{Proof of subclaim 1:} Let $t=\gcd(\omega, h)$. Then $\omega \FF[x] + h\FF[x]=t\FF[x]$ and the equality $\omega f=g'h-h'g$ implies that $h'g\in t\FF[x]$. But $t$ is a divisor of $h$ and $\gcd(h, h')=1$ so it follows that $g\in t\FF[x]$, as required.\hfill $\blacksquare$

\medskip

Take $q, r\in\FF[x]$ with $g=\omega q+rh$. Applying $\varkappa$ to this equality we obtain $\varkappa(g)=\omega \varkappa(q)+\varkappa(rh)$ and thus $\omega$ divides $\varkappa(rh)$. So if suffices to prove that if $\omega$ divides $\varkappa(rh)$ then $rh\in\omega\FF[x] + h\FF[x^p]$. In other words, we may assume without loss of generality that $g=rh$.

Write $r=r_0+r_1$, with $r_0\in\FF[x^p]$ and $r_1\in\bigoplus_{i=1}^{p-1}\FF[x^p]x^i$. As $\varkappa(rh)=\varkappa(r_1 h)$, we may assume that $r_0=0$. So, without loss of generality, we assume that $r\in\bigoplus_{i=1}^{p-1}\FF[x^p]x^i$.

\medskip 

\textit{Subclaim 2:} {\it $\omega$ divides $rh$.}

\smallskip
\textit{Proof of subclaim 2:} Note that $\varkappa(rh)=r'h^2$, so we need to show that if $\omega$ divides $r'h^2$, then $\omega$ divides $rh$. From this point on, our proof follows that of \cite[Lem.\ 6.28\, (iv)]{BLO15ja}, although the details are a bit more intricate and some modifications are needed. Thus, we suspend the proof of the subclaim here and refer the interested reader to the proof of \cite[Lem.\ 6.28\, (iv)]{BLO15ja}.\hfill $\blacksquare$

\smallskip

By the above arguments, the claim is also established, thus proving the Theorem.

\end{proof}

\section{The contracting homotopies $\s_{-1}, \s_0$ and $\s_1$}\label{S:hom}

Recall the definition of the right $\A$-module maps $\s_{-1}$ and $\s_0$, given at the beginning of Section~\ref{S:res}. In this section we prove the two final relations in \eqref{E:res:homotopies}, together with a few other useful identities. 

\begin{lemma}\label{L:sm1s0}
Let $f\in\FF[x]$, $a, b\in\A$ and $\alpha\in\A\ot \V\ot\A$. The following hold:
\begin{enumerate}[label=\textup{(\alph*)}]
\item $\s_0(fa\ot b)=f\s_0(a\ot b)+\s_0(f\ot ab)$.\label{L:sm1s0:2}
\item $\s_0(f\dd_0(\alpha))=f\s_0(\dd_0(\alpha))$.\label{L:sm1s0:4}
\end{enumerate}
\end{lemma}

\begin{proof}
To prove \ref{L:sm1s0:2}, notice that, by the definition of $\s_0$, we have $\s_0(x^k\ot 1)=\sum_{i=0}^{k-1} x^i\ot x\ot x^{k-1-i}$, and similarly for $\s_0(\y^\ell\ot 1)$. Thus, we have $\s_0(x^k\y^\ell\ot 1)=x^k \s_0(\y^\ell\ot 1) + \s_0(x^k\ot \y^\ell)$. It also follows easily that $\s_0(x^{j+k}\ot 1)=x^j \s_0(x^{k}\ot 1)+ \s_0(x^{j}\ot x^{k})$.

Since $\s_0$ is a right $\A$-module map, we can take $b=1$ and by linearity we can further assume that $f=x^j$ and $a=x^k\y^\ell$. Then:
\begin{align*}
f\s_0(a\ot b)+\s_0(f\ot ab) &= x^js_0(x^k\y^\ell \ot 1)+\s_0(x^j \ot x^k\y^\ell)  \\
&=x^j\left(x^k \s_0(\y^\ell\ot 1) + \s_0(x^k\ot \y^\ell)\right)+\s_0(x^j \ot x^k)\y^\ell \\
&=x^{j+k} \s_0(\y^\ell\ot 1) + \left(x^j \s_0(x^k\ot 1)+\s_0(x^j \ot x^k)\right)\y^\ell \\
&=x^{j+k} \s_0(\y^\ell\ot 1) + \s_0(x^{j+k}\ot 1)\y^\ell \\
&=\s_0(x^{j+k}\y^\ell \ot 1)=\s_0(fa\ot b).
\end{align*}

As above, it suffices to prove \ref{L:sm1s0:4} for $\alpha=a\ot v\ot 1$. Using \ref{L:sm1s0:2}, we have:
\begin{align*}
s_0(f\dd_0(a\ot v\ot 1)) &= \s_0(fa(v\ot 1-1\ot v))=\s_0(fav\ot 1-fa\ot v))\\
&=f\s_0(av\ot 1)+\s_0(f\ot av)-f \s_0(a\ot v)-\s_0(f\ot av)\\
&=f\s_0(av\ot 1-a\ot v)=f\s_0(\dd_0(a\ot v\ot 1)).
\end{align*}
\end{proof}

Recall that we have fixed $\rr$ as the basis element of the one-dimensional vector space $\R$. Consider the linear map $G:\FF[x]\longrightarrow\A\ot\R\ot\A$ defined by
\begin{equation}\label{E:def:G}
G(x^k)=\sum_{i=0}^{k-1}x^i\ot \rr\ot x^{k-1-i}, \qquad \mbox{for all $k\geq 0$,}
\end{equation}
with $G(1)=0$. Also, recall that $\delta$ denotes the derivation of $\FF[x]$ defined by $\delta(f) = f'h$, so that $[\y, f]=\delta(f)$, for all $f\in\FF[x]$.

\begin{lemma}\label{L:Gprop}
The map $G$ is a derivation and, for any $f\in\FF[x]$,
\begin{equation*}
\dd_1\circ G(f)=1\ot\y\ot f-f\ot\y\ot 1 -\s_0(f\ot \y) -\s_0(\delta(f)\ot 1)+\y \s_0(f\ot 1). 
\end{equation*}
\end{lemma}

\begin{proof}
Notice that $G(f)=\tau\circ \s_0(f\ot 1)$, where $\tau:\A\ot\V\ot\A\longrightarrow\A\ot\R\ot\A$ is the $\A$-bimodule map which sends both $1\ot x\ot 1$ and $1\ot \y\ot 1$ to $1\ot \rr\ot 1$. Thus, by Lemma~\ref{L:sm1s0}\,\ref{L:sm1s0:2}, $G$ is a derivation.

We deduce that $\dd_1\circ G$ is also a derivation. Define $D: \FF[x]\longrightarrow\A\ot\V\ot\A$ by $D(f)=1\ot\y\ot f-f\ot\y\ot 1 -\s_0(f\ot \y) -\s_0(\delta(f)\ot 1)+\y \s_0(f\ot 1)$. To prove the claimed identiy, it suffices to show that $D$ is also a derivation and that $\dd_1\circ G(x)=D(x)$. The latter is easy to verify, so we turn to proving that $D$ is a derivation, which is also straightforward, using the properties of $\s_0$:
\begin{align*}
D(fg)&= 1\ot\y\ot fg-fg\ot\y\ot 1 -\s_0(fg\ot \y) -\s_0(\delta(fg)\ot 1)+\y \s_0(fg\ot 1)\\
&= 1\ot\y\ot fg-fg\ot\y\ot 1 -f\s_0(g\ot \y) -\s_0(f\ot g\y)-\s_0(\delta(f)g\ot 1)\\*
& \qquad -\s_0(f\delta(g)\ot 1)+\y f \s_0(g\ot 1)+\y \s_0(f\ot g)\\
&= 1\ot\y\ot fg-fg\ot\y\ot 1 -f\s_0(g\ot \y) -\s_0(f\ot \y g)+\cancel{\s_0(f\ot \delta(g))}\\*
& \qquad -\cancel{\delta(f)\s_0(g\ot 1)}-\s_0(\delta(f)\ot g) -f \s_0(\delta(g)\ot 1) -\cancel{\s_0(f\ot \delta(g))}\\*
& \qquad+f\y \s_0(g\ot 1)+\cancel{\delta(f) \s_0(g\ot 1)}+\y \s_0(f\ot g)\\
&=fD(g)+D(f)g.
\end{align*}%
\end{proof}

We are finally ready to define the homotopy $\s_1:\A\ot\V\ot\A\longrightarrow\A\ot\R\ot\A$. This is the right $\A$-module map defined inductively as follows, for $f\in\FF[x]$, $a, b\in\A$ and $\ell\geq 0$:
\begin{itemize}
\item $\s_1(a\ot\y\ot b)=0$;
\item $\s_1(f\y^\ell\ot x\ot a)=f \s_1(\y^\ell\ot x\ot 1)a$;
\item $\s_1(1\ot x\ot 1)=0$;
\item $\s_1(\y^{\ell+1}\ot x\ot 1)=\y \s_1(\y^{\ell}\ot x\ot 1)+\sum_{j=0}^\ell {\ell\choose j} \left(G\circ\delta^j(x)\right) \y^{\ell-j}$, where $\delta(f) = f'h$ and $G$ is the linear map given by \eqref{E:def:G}.
\end{itemize}

\begin{lemma}\label{L:rel2}
The map $\s_1$ satisfies $\s_{0}\circ\dd_0 + \dd_1\circ \s_1 =1_{\A\ot\V\ot\A}$.
\end{lemma}
\begin{proof}
We start by showing that the claimed equality holds for elements of the form $\y^\ell\ot x\ot 1$, by induction on $\ell\geq 0$. In case $\ell=0$ we have
\begin{equation*}
(\s_{0}\circ\dd_0+ \dd_1\circ \s_1) (1\ot x\ot 1)=\s_{0}(x\ot 1-1\ot x)=1\ot x\ot 1. 
\end{equation*}

Next, assume that the result holds for elements of the form $\y^k\ot x\ot 1$, with $k\leq \ell$. Using \eqref{eq:Ahcom} we have

\begin{align*}
\s_{0}(\dd_0 & (\y^{\ell+1} \ot x\ot 1))=\s_{0}(\y^{\ell+1} x\ot 1-\y^{\ell+1}\ot x)\nonumber\\
&=\sum_{j=0}^{\ell+1}{\ell+1\choose j}\s_{0}(\delta^j(x)\y^{\ell+1-j}\ot 1)-\sum_{k=0}^\ell \y^k\ot\y\ot\y^{\ell-k}x\nonumber\\
&=\sum_{j=0}^{\ell+1}{\ell+1\choose j}\delta^j(x) \s_{0}(\y^{\ell+1-j}\ot 1)
 +\sum_{j=0}^{\ell+1}{\ell+1\choose j}\s_{0}(\delta^j(x)\ot 1)\y^{\ell+1-j}\nonumber\\*
& \qquad -\sum_{k=0}^\ell \y^k\ot\y\ot\y^{\ell-k}x\nonumber\\
&=\sum_{j=0}^{\ell+1}{\ell+1\choose j}\delta^j(x) \s_{0}(\y^{\ell+1-j}\ot 1)
 +\sum_{j=0}^{\ell}{\ell\choose j}\s_{0}(\delta^j(x)\ot 1)\y^{\ell+1-j}\nonumber\\*
& \qquad+\sum_{j=1}^{\ell+1}{\ell\choose j-1}\s_{0}(\delta^j(x)\ot 1)\y^{\ell+1-j}
 -\sum_{k=0}^\ell \y^k\ot\y\ot\y^{\ell-k}x\nonumber\\ 
&=\sum_{j=0}^{\ell+1}{\ell+1\choose j}\delta^j(x) \s_{0}(\y^{\ell+1-j}\ot 1)
 +\sum_{j=0}^{\ell}{\ell\choose j}\s_{0}(\delta^j(x)\ot 1)\y^{\ell+1-j} \\*
& \qquad+\sum_{j=0}^{\ell}{\ell\choose j}\s_{0}(\delta^{j+1}(x)\ot 1)\y^{\ell-j}
 -\sum_{k=0}^\ell \y^k\ot\y\ot\y^{\ell-k}x.\nonumber
\end{align*}
Also, by the inductive definition of $\s_1$ and the fact that $\dd_1$ is a bimodule map, $\dd_1 (\s_1(\y^{\ell+1} \ot x\ot 1))=\y\dd_1 (\s_1(\y^{\ell}\ot x\ot 1))+\sum_{j=0}^\ell {\ell\choose j} (\dd_1 \circ G)(\delta^j(x))\y^{\ell-j}$. By the induction hypothesis we have
\begin{align*}
\y\dd_1 (\s_1 & (\y^{\ell}\ot x\ot 1))= \y^{\ell+1}\ot x\ot 1 - \y \s_0(\dd_0(\y^{\ell}\ot x\ot 1))\\
&= \y^{\ell+1}\ot x\ot 1 - \y \s_0(\y^{\ell} x\ot 1) + \y \s_0(\y^{\ell}\ot x) \\
&= \y^{\ell+1}\ot x\ot 1 - \y \sum_{j=0}^{\ell}{\ell\choose j}\s_0(\delta^j(x)\y^{\ell-j} \ot 1) + \y \s_0(\y^{\ell}\ot x) \\
&= \y^{\ell+1}\ot x\ot 1 - \y \sum_{j=0}^{\ell}{\ell\choose j} \delta^j(x) \s_0(\y^{\ell-j} \ot 1) \\*
& \qquad -\y \sum_{j=0}^{\ell}{\ell\choose j}\s_0(\delta^j(x)\ot\y^{\ell-j})
 +\sum_{k=0}^{\ell-1}\y^{k+1}\ot\y\ot \y^{\ell-1-k}x \\
&= \y^{\ell+1}\ot x\ot 1 -  \sum_{j=0}^{\ell}{\ell\choose j} \delta^j(x)\y \s_0(\y^{\ell-j} \ot 1) \\*
& \qquad -\sum_{j=0}^{\ell}{\ell\choose j} \delta^{j+1}(x) \s_0(\y^{\ell-j} \ot 1)  -\sum_{j=0}^{\ell}{\ell\choose j}\y \s_0(\delta^j(x)\ot\y^{\ell-j})\\*
& \qquad +\sum_{k=0}^{\ell-1}\y^{k+1}\ot\y\ot \y^{\ell-1-k}x. 
\end{align*}
Finally, using Lemma~\ref{L:Gprop}, it follows that
\begin{align*}
\sum_{j=0}^\ell {\ell\choose j} & (\dd_1 \circ G)(\delta^j(x))\y^{\ell-j} =
1\ot\y\ot \sum_{j=0}^\ell {\ell\choose j} \delta^j(x)\y^{\ell-j}\\*
& \qquad -\sum_{j=0}^\ell {\ell\choose j} \delta^j(x)\ot\y\ot \y^{\ell-j}
-\sum_{j=0}^\ell {\ell\choose j} \s_0(\delta^j(x)\ot\y)\y^{\ell-j}\\*
& \qquad -\sum_{j=0}^\ell {\ell\choose j} \s_0(\delta^{j+1}(x)\ot 1)\y^{\ell-j}
+\sum_{j=0}^\ell {\ell\choose j} \y \s_0(\delta^j(x)\ot 1)\y^{\ell-j}\\
&= 1\ot\y\ot \y^\ell x -\sum_{j=0}^\ell {\ell\choose j} \delta^j(x)\ot\y\ot \y^{\ell-j}\\*
& \qquad 
-\sum_{j=0}^\ell {\ell\choose j} \s_0(\delta^j(x)\ot 1)\y^{\ell+1-j}
-\sum_{j=0}^\ell {\ell\choose j} \s_0(\delta^{j+1}(x)\ot 1)\y^{\ell-j}\\*
& \qquad 
+\sum_{j=0}^\ell {\ell\choose j} \y \s_0(\delta^j(x)\ot 1)\y^{\ell-j}.
\end{align*}

The term $\displaystyle \sum_{j=0}^{\ell+1}{\ell+1\choose j}\delta^j(x) \s_{0}(\y^{\ell+1-j}\ot 1)$ in the expression for $\displaystyle \s_{0}(\dd_0 (\y^{\ell+1} \ot x\ot 1))$ can be further expanded as follows:

\begin{align*}
\sum_{j=0}^{\ell+1}{\ell+1\choose j} & \delta^j(x) \s_{0}(\y^{\ell+1-j}\ot 1) = 
\sum_{j=0}^{\ell}{\ell \choose j}\delta^j(x) \s_{0}(\y^{\ell+1-j}\ot 1)\\*
&\qquad +\sum_{j=1}^{\ell+1}{\ell\choose j-1}\delta^j(x) \s_{0}(\y^{\ell+1-j}\ot 1)\\
&=\sum_{j=0}^{\ell}{\ell \choose j}\delta^j(x) \y \s_{0}(\y^{\ell-j}\ot 1)
 + \sum_{j=0}^{\ell}{\ell \choose j}\delta^j(x) \ot\y\ot \y^{\ell-j}\\*
&\qquad +\sum_{j=0}^{\ell}{\ell\choose j}\delta^{j+1}(x) \s_{0}(\y^{\ell-j}\ot 1).
\end{align*}

Combining all of these expressions, we see easily that all terms cancel out except for the term $\y^{\ell+1}\ot x\ot 1$ in the expansion of $\y\dd_1 (\s_1  (\y^{\ell}\ot x\ot 1))$ above, so we have the desired identity $(\s_{0}\circ\dd_0 + \dd_1\circ \s_1)(\y^{\ell+1}\ot x\ot 1) =\y^{\ell+1}\ot x\ot 1$, establishing the inductive step.  

By the equality $\s_0(f\dd_0(\alpha))=f\s_0(\dd_0(\alpha))$, for $f\in\FF[x]$ and $\alpha\in\A\ot \V\ot\A$, proved in Lemma~\ref{L:sm1s0}, and the definition of $\s_1$, we conclude that $(\s_{0}\circ\dd_0 + \dd_1\circ \s_1)(f\y^{\ell}\ot x\ot a) =f\y^{\ell}\ot x\ot a$, for all $\ell\geq 0$, $f\in\FF[x]$ and $a\in\A$. So next we focus on elements of the form $f\y^{\ell}\ot \y\ot a$. We will make use of the identity $\s_0(\y^{\ell+1}\ot 1- \y^{\ell}\ot \y)=\y^\ell\ot\y\ot 1$ to perform the required calculation. Then,
\begin{align*}
 (\s_{0}\circ\dd_0 + \dd_1\circ \s_1)(f\y^{\ell}\ot \y\ot a) &=\s_{0}(\dd_0(f\y^{\ell}\ot \y\ot a))
 =\s_{0}(f\dd_0(\y^{\ell}\ot \y\ot 1))a\\
 &=f\s_{0}(\y^{\ell+1}\ot 1-\y^{\ell}\ot \y)a=f\y^\ell\ot\y\ot a.
\end{align*}

Combining all of the above, we have proved the claim.
\end{proof}

Now we aim to prove the last relation in \eqref{E:res:homotopies}, namely 
$\s_1\circ \dd_1 =1_{\A\ot\R\ot\A}$. We start with a technical identity which just depends on the fact that $G$ and $\delta$ are derivations.

\begin{lemma}\label{L:longid}
Given $k\geq 1$ and $r\geq 0$,
\begin{equation*}
\sum_{i=0}^{k-1}\sum_{j=0}^{r}\sum_{t=0}^{r-j} {r\choose j}{r-j\choose t} \delta^j(x^i)G(\delta^t(x))\delta^{r-j-t}(x^{k-i-1})=G(\delta^r(x^k)).
\end{equation*}
\end{lemma}

\begin{proof}
First, fix $0\leq i\leq k-1$. Using the change of variables $m=j+t$, the combinatorial identity $\displaystyle {r\choose j}{r-j\choose m-j}={r\choose m}{m\choose j}$, and the derivation property of $G$, we have
\begin{align*}
\sum_{j=0}^{r} \sum_{t=0}^{r-j} & {r\choose j}{r-j\choose t} \delta^j(x^i)G(\delta^t(x))\delta^{r-j-t}(x^{k-i-1}) \nonumber\\
&= \sum_{m=0}^{r}\sum_{j=0}^{m} {r\choose j}{r-j\choose m-j} \delta^j(x^i)G(\delta^{m-j}(x))\delta^{r-m}(x^{k-i-1})\nonumber\\
&= \sum_{m=0}^{r}\sum_{j=0}^{m} {r\choose m}{m\choose j} G(\delta^j(x^i) \delta^{m-j}(x))\delta^{r-m}(x^{k-i-1}) \\*
&\qquad - \sum_{m=0}^{r}\sum_{j=0}^{m} {r\choose j}{r-j\choose m-j} G(\delta^j(x^i))\delta^{m-j}(x)\delta^{r-m}(x^{k-i-1}).
\end{align*}
Now recall that for any derivation $D$, the generalized Leibniz rule says that $D^{\ell}(ab)=\sum _{k=0}^{\ell}{\ell \choose k} D^{k}(a) D^{\ell-k}(b)$. So the right-hand side of the running equality is
\begin{align*}
\sum_{m=0}^{r} & {r\choose m} G(\delta^m(x^{i+1}))\delta^{r-m}(x^{k-i-1}) \\*
&\qquad - \sum_{j=0}^{r}{r\choose j}G(\delta^j(x^i))\sum_{m=j}^{r} {r-j\choose m-j} \delta^{m-j}(x)\delta^{r-m}(x^{k-i-1})\\
&= \sum_{m=0}^{r}  {r\choose m} G(\delta^m(x^{i+1}))\delta^{r-m}(x^{k-i-1}) 
 - \sum_{j=0}^{r}{r\choose j}G(\delta^j(x^i))\delta^{r-j}(x^{k-i})\\
&= \sum_{m=0}^{r}  {r\choose m}\left( G(\delta^m(x^{i+1}))\delta^{r-m}(x^{k-i-1}) 
 - G(\delta^m(x^i))\delta^{r-m}(x^{k-i}) \right).
\end{align*}
Summing over all  $0\leq i\leq k-1$, we obtain
\begin{align*}
\sum_{m=0}^{r}  {r\choose m} & \sum_{i=0}^{k-1} \left( G(\delta^m(x^{i+1}))\delta^{r-m}(x^{k-i-1}) 
 - G(\delta^m(x^i))\delta^{r-m}(x^{k-i}) \right)\\
 &=\sum_{m=0}^{r}  {r\choose m} \left( G(\delta^m(x^{k}))\delta^{r-m}(1) 
 - G(\delta^m(1))\delta^{r-m}(x^{k}) \right).
\end{align*}
Since $G(\delta^m(1))=0$ for all $m\geq 0$ and $\delta^{r-m}(1)=0$ for all $m<r$, the latter expression is just $G(\delta^r(x^{k}))$, as desired.
\end{proof}

Our next results concern the computation of $\s_{1}$.
 
\begin{prop}\label{P:recsoneszero}
For all $\ell\geq 0$ and all $f\in\FF[x]$, the following identity holds:
\begin{equation*}
\s_{1}(\y^{\ell+1}\s_0(f\ot 1))= \y \s_{1}(\y^{\ell}\s_0(f\ot 1))+\sum_{j=0}^\ell {\ell\choose j} G(\delta^j(f))\y^{\ell-j}.
\end{equation*}
\end{prop}
\begin{proof}
By linearity, it is enough to show the identity
\begin{equation*}
\s_{1}(\y^{\ell+1}\s_0(x^k\ot 1))= \y \s_{1}(\y^{\ell}\s_0(x^k\ot 1))+\sum_{j=0}^\ell {\ell\choose j} G(\delta^j(x^k))\y^{\ell-j},
\end{equation*}
for all $k\geq 0$. This holds trivially if $k=0$, so we assume that $k\geq 1$.

Firstly, let us observe that by the relation $\y f=f\y+\delta(f)$ and the recurrence relation for $\s_1$, it follows that 
\begin{equation*}
\s_1(\y f\y^\ell\ot x\ot 1)=\y \s_1(f\y^\ell\ot x\ot 1)+\sum_{t=0}^\ell {\ell\choose t} f G(\delta^t(x))\y^{\ell-t},
\end{equation*} 
for all $f\in\FF[x]$ and $\ell\geq 0$. Thus, using \eqref{eq:Ahcom}, we have, for $0\leq i\leq k-1$:
\begin{align*}
\s_1  (\y^{\ell+1}  x^i & \ot x\ot 1) = 
\sum_{j=0}^{\ell} {\ell\choose j}\s_1(\y \delta^j(x^i)\y^{\ell-j} \ot x\ot 1)\\
&=\y \sum_{j=0}^{\ell} {\ell\choose j}\s_1(\delta^j(x^i)\y^{\ell-j} \ot x\ot 1)\\*
&\qquad +\sum_{j=0}^{\ell} {\ell\choose j}\sum_{t=0}^{\ell-j} {\ell-j\choose t} \delta^j(x^i)
 G(\delta^t(x))\y^{\ell-j-t}\\
&=\y \s_1(\y^{\ell}x^i \ot x\ot 1)
 +\sum_{j=0}^{\ell}\sum_{t=0}^{\ell-j} {\ell\choose j} {\ell-j\choose t} \delta^j(x^i)
 G(\delta^t(x))\y^{\ell-j-t}.
\end{align*}

Hence, 
\begin{align*}
\s_{1}(\y^{\ell+1}  \s_0(x^k\ot 1)) &= \sum_{i=0}^{k-1} \s_1  (\y^{\ell+1}  x^i  \ot x\ot x^{k-i-1})\\
&=\y \s_{1}(\y^{\ell}\s_0(x^k\ot 1))\\*
&\qquad + \sum_{i=0}^{k-1}\sum_{j=0}^{\ell}\sum_{t=0}^{\ell-j} {\ell\choose j} {\ell-j\choose t} \delta^j(x^i)
 G(\delta^t(x))\y^{\ell-j-t}x^{k-i-1},
\end{align*}
and it remains to prove that $\displaystyle \sum_{i=0}^{k-1}\sum_{j=0}^{\ell}\sum_{t=0}^{\ell-j} {\ell\choose j} {\ell-j\choose t} \delta^j(x^i) G(\delta^t(x))\y^{\ell-j-t}x^{k-i-1}$ is equal to 
$\displaystyle \sum_{j=0}^\ell {\ell\choose j} G(\delta^j(x^k))\y^{\ell-j}$.

Using \eqref{eq:Ahcom}, we can write the former as 
\begin{equation*}
\sum_{i=0}^{k-1}\sum_{j=0}^{\ell}\sum_{t=0}^{\ell-j}\sum_{m=0}^{\ell-j-t} {\ell\choose j} {\ell-j\choose t} {\ell-j-t \choose m}\delta^j(x^i) G(\delta^t(x))\delta^m(x^{k-i-1})\y^{\ell-j-t-m}.
\end{equation*}
Let $a=j+t+m$. Notice that $0\leq a\leq \ell$ and that the sum above can be written as $\displaystyle \sum_{a=0}^{\ell}\phi(i, j, t) \y^{\ell-a}$, where 
\begin{equation*}
\phi(i, j, t)= \sum_{i=0}^{k-1}\sum_{j=0}^{a}\sum_{t=0}^{a-j} {\ell\choose j} {\ell-j\choose t} {\ell-j-t \choose a-j-t}\delta^j(x^i) G(\delta^t(x))\delta^{a-j-t}(x^{k-i-1}).
\end{equation*}
Therefore, we just need to prove that $\phi(i, j, t)={\ell\choose a}G(\delta^a(x^k))$. Since 
$$
{\ell\choose j} {\ell-j\choose t} {\ell-j-t \choose a-j-t}={\ell\choose a}{a\choose j}{a-j\choose t},
$$
by Lemma~\ref{L:longid} we deduce that 
\begin{align*}
\phi(i, j, t)&= {\ell\choose a}\sum_{i=0}^{k-1}\sum_{j=0}^{a}\sum_{t=0}^{a-j}  {a\choose j}{a-j\choose t} \delta^j(x^i) G(\delta^t(x))\delta^{a-j-t}(x^{k-i-1})\\
&={\ell\choose a}G(\delta^a(x^k)).
\end{align*}
Hence, the result is established.
\end{proof}

We are now able to determine closed formulas for $\s_{1}(\y^{\ell+1}\s_0(f\ot 1))$ and $\s_{1}(\y^{\ell+1}\ot x\ot 1)$.

\begin{prop}\label{P:form:s1}
For all $\ell\geq 0$ and $f\in\FF[x]$, we have:
\begin{equation*}
\s_{1}(\y^{\ell+1}\s_0(f\ot 1))= \sum_{j=0}^{\ell}\sum_{k=0}^{\ell-j} {\ell-k\choose j} \y^k G(\delta^j(f))\y^{\ell-j-k}.
\end{equation*}
In particular, taking $f=x$, we obtain the following explicit formula for $\s_{1}$:
\begin{equation*}
\s_{1}(\y^{\ell+1}\ot x\ot 1)= \sum_{j=0}^{\ell}\sum_{k=0}^{\ell-j} {\ell-k\choose j} \y^k G(\delta^j(x))\y^{\ell-j-k}.
\end{equation*}
\end{prop}

\begin{proof}
If $\ell=0$, Proposition~\ref{P:recsoneszero} yields $\s_{1}(\y \s_0(f\ot 1))=G(f)$, which agrees with the formula we are proving. We proceed inductively, using Proposition~\ref{P:recsoneszero}:
\begin{align*}
\s_1(\y^{\ell+1}\s_0(f\ot 1))&=\y \s_1(\y^{\ell}\s_0(f\ot 1))+\sum_{j=0}^\ell {\ell\choose j} G\circ\delta^j(f)\y^{\ell-j} \\
&=\sum_{j=0}^{\ell-1}\sum_{k=0}^{\ell-j-1} {\ell-k-1\choose j} \y^{k+1} G(\delta^j(f))\y^{\ell-j-k-1}\\*
&\qquad +\sum_{j=0}^\ell {\ell\choose j} G\circ\delta^j(f)\y^{\ell-j}\\
&=\sum_{j=0}^{\ell-1}\sum_{k=1}^{\ell-j} {\ell-k\choose j} \y^{k} G(\delta^j(f))\y^{\ell-j-k}\\*
&\qquad +\sum_{j=0}^{\ell-1} {\ell\choose j} G\circ\delta^j(f)\y^{\ell-j} + {\ell\choose\ell} G\circ\delta^\ell(f)\\
&=\sum_{j=0}^{\ell}\sum_{k=0}^{\ell-j} {\ell-k\choose j} \y^{k} G(\delta^j(f))\y^{\ell-j-k}.
\end{align*}
\end{proof}

Finally, we can prove the main result of this section.

\begin{thm}\label{T:hom:main}
The right $\A$-module maps $\s_{-1}, \s_0$ and $\s_1$ form a contracting homotopy for \eqref{E:res:K}.
\end{thm}
\begin{proof}
It remains to prove the identity $\s_1\circ \dd_1 =1_{\A\ot\R\ot\A}$ from \eqref{E:res:homotopies}, and it clearly suffices to check this identity on elements of the form $\y^\ell \ot \rr\ot 1$, as $\s_1$ is also a left $\FF[x]$-module homomorphism. The case $\ell=0$ is straightforward,
so assume that $\ell\geq 1$. Then
\begin{align*}
\s_1(\dd_1(\y^\ell\ot \rr\ot 1)) &=\s_1(\y^\ell\dd_1(1\ot \rr\ot 1))\\
&=\s_1(\y^{\ell+1}\ot x\ot 1)-\s_1(\y^\ell\ot x\ot 1)\y-\s_1(\y^\ell \s_0(\delta(x)\ot 1)),
\end{align*}
and by Proposition~\ref{P:form:s1}, we have
\begin{align*}
\s_1(\y^{\ell+1}\ot x\ot 1) &= \sum_{j=0}^{\ell}\sum_{k=0}^{\ell-j} {\ell-k\choose j} \y^k G(\delta^j(x))\y^{\ell-j-k}.
\end{align*}
Using adequate combinatorial identities, we obtain
\begin{align*}
\s_1(\y^{\ell+1}\ot x\ot 1) &= \sum_{j=1}^{\ell-1}\sum_{k=0}^{\ell-j-1} {\ell-k-1\choose j} \y^k G(\delta^j(x))\y^{\ell-j-k}\\*
&\qquad + \sum_{j=1}^{\ell}\sum_{k=0}^{\ell-j} {\ell-k-1\choose j-1} \y^k G(\delta^j(x))\y^{\ell-j-k}\\*
&\qquad + \sum_{k=0}^{\ell-1} {\ell-k-1\choose 0} \y^k G(\delta^0(x))\y^{\ell-k}\\*
&\qquad +\y^\ell G(x)\\
&= \sum_{j=0}^{\ell-1}\sum_{k=0}^{\ell-j-1} {\ell-k-1\choose j} \y^k G(\delta^j(x))\y^{\ell-j-k}\\*
&\qquad + \sum_{j=0}^{\ell-1}\sum_{k=0}^{\ell-j-1} {\ell-k-1\choose j} \y^k G(\delta^{j+1}(x))\y^{\ell-j-k-1}\\*
&\qquad +\y^\ell \ot \rr\ot 1 \\
&=\s_1(\y^\ell\ot x\ot 1)\y + \s_1(\y^\ell \s_0(\delta(x)\ot 1))+\y^\ell \ot \rr\ot 1,
\end{align*}
which proves the desired identity.
\end{proof}

\section{The Gerstenhaber bracket: general remarks}\label{S:Gerst}

The Hochschild cohomology $\hoch^{\bullet}(\A)=\bigoplus_{n\geq 0}\hoch^n(\A)$ has a rich structure, including an associative, graded-commutative product (relative to homological degree), given by the cup product, and also a graded Lie bracket $\gb{\, ,}$ of (homological) degree $-1$; these are related by the graded Poisson identity. In particular, the graded anti-symmetric property of $\gb{\, ,}$ means 
\begin{equation*}
\gb{\alpha, \beta}= -(-1)^{(m-1)(n-1)}\gb{\beta, \alpha}, \quad \mbox{for all $\alpha\in\hoch^m(\A)$ and $\beta\in\hoch^n(\A)$,} 
\end{equation*}
and there is a corresponding graded version of the Jacobi identity (see \cite{mG63}). Under this construction, $\hoch^{\bullet}(\A)$ becomes a \textit{Gerstenhaber algebra}. In particular, the Jacobi identity implies that $\hoch^{\bullet}(\A)$ is a Lie module for the Lie algebra $\hoch^{1}(\A)$, extending the usual Lie bracket of derivations on $\hoch^{1}(\A)$. 
In case $A$ is a smooth finitely-generated $\FF$-algebra  and $\FF$ is perfect, the Hochschild-Kostant-Rosenberg Theorem gives an isomorphism of Gerstenhaber algebras, telling that, in this situation, the Gerstenhaber bracket is the generalization to higher degrees of the Schouten-Nijenhuis bracket.

The Gerstenhaber structure of Hochschild cohomology is particularly interesting for us since in case $\chara(\FF)=0$ and $\gcd(h, h')\neq 1$, the description of $\hoch^{1}(\A)$ involves the Witt algebra $\W$. 
In prime characteristic, most of the computations of the Gerstenhaber structure in Hochschild cohomology concern group algebras and tame blocks, see for example \cite{BKL, S}.

Although the Gerstenhaber bracket does not depend on the chosen bimodule projective resolution of $\A$, it is in general difficult to compute it on an arbitrary resolution other than the bar resolution. In spite of this, we always have $\gb{\barr{D}, z}=\barr{D}(z)$ and $\gb{\barr{D}, \barr{D'}}=\barr{\gb{D, D'}}$ for $D, D'\in\der(\A)$ and $z\in\cent{}$, so it remains to compute $\gb{\hoch^{1}(\A), \hoch^{2}(\A)}$, which is what we undertake in this section. 
Notice that, in our case, we already have the contracting homotopy of the minimal resolution, from which the comparison maps can be obtained. Nevertheless, we will use an easier method that, for the family of algebras we consider, also needs the contracting homotopy.

To avoid cumbersome notation, we identify $D\in\der(\A)$ with its canonical image $\barr{D}\in\hoch^1(\A)$. We will often refer to the map $\gb{D, -}:\hoch^i(\A)\longrightarrow\hoch^i(\A)$ as the (Lie) action of $D\in\hoch^1(\A)$ on $\hoch^i(\A)$.

\subsection{The method of Su\'arez-\'Alvarez for computing $\gb{\hoch^{1}(\A), -}$}\label{SS:gb:MSAmethod}

In this subsection, we will describe a method devised by Su\'arez-\'Alvarez in \cite{mSA17} to compute the Gerstenhaber bracket $\gb{\hoch^{1}(\A), -}$ in terms of an arbitrary projective resolution of $\A$ as a bimodule. The reader is advised to consult \cite{mSA17} for further details and all the proofs.

Fix an $\FF$-algebra $\BB$ and a derivation $\psi:\BB\longrightarrow\BB$. Given a left $\BB$-module $M$, we say that a linear map $f:M\longrightarrow M$ satisfying $f(bm) = bf(m) + \psi(b)m$ for all $b\in\BB$ and $m\in M$ is a \textit{$\psi$-operator} on $M$. Given a projective resolution 
\begin{equation*}
\begin{tikzcd}[row sep=5ex, column sep = 4em]
\cdots\arrow[r] & P_2 \arrow[r, "d_2"]  & P_1 \arrow[r, "d_1"]  & P_0 \arrow[r, "\epsilon"]  &M \arrow[r]  & 0
\end{tikzcd}
\end{equation*}
of $M$, a \textit{$\psi$-lifting} of the $\psi$-operator $f$ to $P_\bullet$ is a sequence $f_\bullet=\left( f_i\right)_{i\geq 0}$ of $\psi$-operators $f_i:P_i\longrightarrow P_i$ such that the following diagram commutes:

\begin{equation*}
\begin{tikzcd}[row sep=5ex, column sep = 4em]
\cdots\arrow[r] & P_2 \arrow[r, "d_2"] \arrow[d, "f_2"] & P_1 \arrow[r, "d_1"] \arrow[d, "f_1"] & P_0 \arrow[r, "\epsilon"] \arrow[d, "f_0"] &M \arrow[r] \arrow[d, "f"]  & 0\\
\cdots\arrow[r] & P_2 \arrow[r, "d_2"]  & P_1 \arrow[r, "d_1"]  & P_0 \arrow[r, "\epsilon"]  &M \arrow[r]  & 0.
\end{tikzcd}
\end{equation*}
It was shown in \cite[Lem.\ 1.4]{mSA17} that every $\psi$-operator $f$ admits a unique (up to $\BB$-module homotopy) $\psi$-lifting.

Given a $\psi$-operator $f$ and a $\psi$-lifting $f_\bullet$ of $f$ to $P_\bullet$, define a sequence $f^{\sharp}_\bullet=\left( f^\sharp_i\right)_{i\geq 0}$ of linear maps $f^\sharp_i:\Hom_\BB(P_i, M)\longrightarrow\Hom_\BB(P_i, M)$ by
\begin{equation*}
f^\sharp_i(\phi)(p)=f(\phi(p))-\phi(f_i(p)), 
\end{equation*}
for $\phi\in\Hom_\BB(P_i, M)$ and $p\in P_i$. In fact, $f^{\sharp}_\bullet$ is an endomorphism of the complex of vector spaces $\Hom_\BB(P_\bullet, M)$ and the induced map on cohomology 
\begin{equation*}
\nabla^\bullet_{f,P_\bullet}: \mathsf{H}(\Hom_\BB(P_\bullet,M)) \longrightarrow \mathsf{H}(\Hom_\BB(P_\bullet,M)) 
\end{equation*}
depends only on $f$ and not on the choice of $\psi$-lifting $f_\bullet$. What's more, noticing that $\mathsf{H}(\Hom_\BB(P_\bullet,M))$ is canonically isomorphic to $\mathsf{Ext}^{\bullet}_\BB(M, M)$, we obtain a canonical morphism of graded vector spaces
\begin{equation*}
\nabla^\bullet_{f}: \mathsf{Ext}^{\bullet}_\BB(M, M) \longrightarrow \mathsf{Ext}^{\bullet}_\BB(M, M) 
\end{equation*}
which depends only on $f$ and not on the chosen projective resolution of $M$ (see \cite[Thm.\ A]{mSA17}).

Returning to the problem at hand, which is the computation of the bracket $\gb{\hoch^{1}(\A), -}$ in terms of a chosen bimodule projective resolution $\mu:P_\bullet\twoheadrightarrow \A$ of $\A$, set $\BB=\A^{e}$ and $M=\A$, so that $\mu:P_\bullet\twoheadrightarrow \A$ can be identified with a projective resolution of $\A$ as a left $\BB$-module. Given a derivation $D$ of $\A$, construct a new derivation $D^e=D\ot 1_\A+1_\A\ot D$ of $\BB$. It can be readily seen that $D$ is a $D^e$-operator on $\A$. Since $\mathsf{Ext}^{\bullet}_\BB(\A, \A)$ is naturally identified with the Hochschild cohomology $\hoch^{\bullet}(\A)$, the above construction yields a map $\nabla^\bullet_{D}: \hoch^{\bullet}(\A) \longrightarrow \hoch^{\bullet}(\A)$, which by \cite[Sec.\ 2.2]{mSA17} turns out to be $\gb{D, -}$ and which can be computed using any bimodule projective resolution of $\A$, provided that a $D^e$-lifting $D_\bullet$ of $D$ to the given resolution is found.

Going back to the case under study, with $\A=\A_h$, $\epsilon=\mu$ (the multiplication map), $P_0=\A\ot\A$, $P_1=\A\ot\V\ot\A$ and $P_2=\A\ot\R\ot\A$, it can be checked that $D\circ\mu=\mu\circ D^e$ and $D^e$ is trivially a $D^e$-operator on $\A\ot\A$, so we can choose $D_0=D^e$. Taking $i=2$ and using the map $\rho_2$ from Section~\ref{S:res} to identify $\hoch^2(\A)$ with a homomorphic image of $\A$, we obtain the formula describing the Lie action of $\hoch^1(\A)$ on $\hoch^2(\A)$:
\begin{equation}\label{E:SS:MSAmethod:gb1}
\gb{D, a}=D(a)-\chi_a(D_2(1\ot\rr\ot 1)), 
\end{equation}
for $a\in \A$ and $D\in\der(\A)$, where $\chi_a\in \Hom_{\A^e}(\A\ot\R\ot\A, \A)$ is defined by $\chi_a(1\ot\rr\ot 1)=a$.

\subsection{The $D^e$-lifting of $D$ to \eqref{E:res:K}}\label{SS:gb:lifting D_2}

In order to make use of \eqref{E:SS:MSAmethod:gb1}, it remains to determine the $D^e$-lifting $D_2$ of $D$, which we do in this subsection. We begin with a few general observations aimed at simplifying computations, then we determine the $D^e$-liftings $D_1$ and $D_2$.

The proof of the lemma that follows is standard and is thus omitted.

\begin{lemma}\label{L:SS:lifting D_2:general}
Let $\BB$ be an algebra, $\psi:\BB\longrightarrow\BB$ a derivation, $M$ and $N$ left $\BB$-modules, $X\subseteq M$ a generating set for $M$ as a $\BB$-module and $Y\subseteq \BB$ a generating set for $\BB$ as a vector space. 
\begin{enumerate}[label=\textup{(\alph*)}]
\item If $X$ is a free $\BB$-basis for $M$, then for any function $f':X\longrightarrow M$ there is a unique $\psi$-operator $f:M\longrightarrow M$ such that $\eval{f}{X}=f'$.\label{L:SS:lifting D_2:general:1}
\item Let $\phi:M\longrightarrow N$ be a morphism of $B$-modules and let $f:M\longrightarrow M$ and $g:N\longrightarrow N$ be $\psi$-operators. If $\eval{g\circ \phi}{X}=\eval{\phi\circ f}{X}$, then the following square commutes:
\begin{equation*}
\begin{tikzcd}[row sep=5ex, column sep = 4em]
 M \arrow[r, "\phi"] \arrow{d}[left]{f} &N \arrow[d, "g"] \\
M \arrow[r, "\phi"] &N.
\end{tikzcd}
\end{equation*}
\label{L:SS:lifting D_2:general:2}
\item If $f:M\longrightarrow M$ is a linear map such that $f(bm) = bf(m) + \psi(b)m$ for all $b\in Y\subseteq \BB$ and all $m\in X\subseteq M$, then $f$ is a $\psi$-operator.
\label{L:SS:lifting D_2:general:3}
\end{enumerate}
\end{lemma}

Throughout the rest of this subsection, fix $D\in\der(\A)$ and let $D_0=D^e:\A^{e}\longrightarrow\A^{e}$. Next we define a $D_0$-lifting $D_1:\A\ot\V\ot\A\longrightarrow \A\ot\V\ot\A$ in terms of the homotopy $\s_0$.

\begin{lemma}\label{L:SS:lifting D_1}
Let $D_1(a\ot v\ot b)=a\s_0(D(v)\ot b) + D(a)\ot v\ot b + a\ot v\ot D(b)$, for all $a, b\in\A$ and all $v\in\V=\FF x\oplus\FF\y$. Then, extending linearly to $\A\ot\V\ot\A$, this rule defines a $D_0$-operator such that $D_0\circ\dd_0=\dd_0 \circ D_1$. 
\end{lemma}
\begin{proof}
Define first $D_1(1\ot v\ot 1)=\s_0(D(v)\ot 1)$ for $v\in\{ x, \y\}$. Since $\{ 1\ot x\ot 1, 1\ot \y\ot 1\}$ is a free basis for $\A\ot\V\ot\A$ as an $\A^e$-module, Lemma~\ref{L:SS:lifting D_2:general}\ref{L:SS:lifting D_2:general:1} guarantees the existence of a unique $D_0$-operator, which we still denote by $D_1$, defined on $\A\ot\V\ot\A$ and extending the above rule. 

First, notice that by linearity of $D$ and $\s_0$, one has $D_1(1\ot v\ot 1)=\s_0(D(v)\ot 1)$ for all $v\in\V$. Given $a, b\in\A$, the definition of a $D_0$-operator implies that 
\begin{align*}
D_1(a\ot v\ot b)&=D_1((a\ot b)(1\ot v\ot 1))\\ &= (a\ot b)D_1(1\ot v\ot 1) + D_0(a\ot b)(1\ot v\ot 1)\\ &=a\s_0(D(v)\ot 1)b+D(a)\ot v\ot b +a\ot v\ot D(b).
\end{align*}
As $\s_0$ is a right $\A$-module map, this expression matches the one in the statement.

Now, by Lemma~\ref{L:SS:lifting D_2:general}\,\ref{L:SS:lifting D_2:general:2}, it suffices to check the equality $D_0\circ\dd_0=\dd_0 \circ D_1$ on elements of the form $1\ot v\ot 1$. Thus, using the second identity in \eqref{E:res:homotopies},
we establish the final claim:
\begin{align*}
\dd_0 \circ D_1(1\ot v\ot 1)&=\dd_0(\s_0(D(v)\ot 1))=D(v)\ot 1 - \s_{-1}\circ\mu(D(v)\ot 1)\\
&=D(v)\ot 1 - \s_{-1}(D(v))=D(v)\ot 1 - 1\ot D(v)\\ &=D_0(v\ot 1-1\ot v)=D_0\circ\dd_0(1\ot v\ot 1).
\end{align*}
\end{proof}

Before we proceed to define the $D_0$-lifting $D_2$, we prove some auxiliary relations which will simplify several expressions, including one for $D_2(1\ot\rr\ot 1)$.

\begin{lemma}\label{L:SS:lifting D_2:simp}
Let $g\in\FF[x]$, $\alpha\in\A\ot\V\ot\A$, $b\in\A$ and $k, \ell\geq 0$. The following hold:
\begin{enumerate}[label=\textup{(\alph*)}]
\item $\s_1(g\alpha)=g\s_1(\alpha)$;\label{L:SS:lifting D_2:simp:1}
\item $\s_1\circ\s_0=0$;\label{L:SS:lifting D_2:simp:2}
\item $\s_1(\y \s_0(g\y^\ell \ot b))=G(g)\y^\ell b$, where $G$ is given in \eqref{E:def:G};\label{L:SS:lifting D_2:simp:3}
\item $\s_1\circ D_1\circ \s_0 (x^k \ot 1)=\sum_{i=1}^{k-1}\s_1(D(x^i)\ot x\ot x^{k-i-1})$, where this sum is understood to be $0$ in case $k\in\{0, 1\}$.\label{L:SS:lifting D_2:simp:4}
\end{enumerate} 
\end{lemma}
\begin{proof}
Both \ref{L:SS:lifting D_2:simp:1} and \ref{L:SS:lifting D_2:simp:2} follow trivially from the definitions, so we proceed to prove \ref{L:SS:lifting D_2:simp:3}. As before, we can assume that $b=1$. Furthermore, using \ref{L:SS:lifting D_2:simp:1}, \ref{L:SS:lifting D_2:simp:2}, Lemma~\ref{L:sm1s0}\ref{L:sm1s0:2}, the definition of $\s_1$ and Proposition~\ref{P:recsoneszero}, we get:
\begin{align*}
\s_1(\y \s_0(g\y^\ell \ot 1))&=\s_1(\y (g\s_0(\y^\ell \ot 1)+\s_0(g \ot \y^\ell))) \\
&=\s_1(g\y\s_0(\y^\ell \ot 1))+g'h\s_1(\s_0(\y^\ell \ot 1))+\s_1(\y\s_0(g \ot 1))\y^\ell\\
&=g\s_1(\y\s_0(\y^\ell \ot 1))+\s_1(\y\s_0(g \ot 1))\y^\ell\\
&=G(g)\y^\ell.
\end{align*}

Finally, for the proof of \ref{L:SS:lifting D_2:simp:4} we have, using the definition of $D_1$, parts \ref{L:SS:lifting D_2:simp:1} and \ref{L:SS:lifting D_2:simp:2} and the definition of $\s_1$:
\begin{align*}
\s_1\circ D_1\circ \s_0 (x^k \ot 1)&= \sum_{i=0}^{k-1}\s_1\circ D_1(x^i\ot x \ot x^{k-i-1})\\
&=\sum_{i=0}^{k-1}\s_1(D(x^i)\ot x\ot x^{k-i-1})\\
&=\sum_{i=1}^{k-1}\s_1(D(x^i)\ot x\ot x^{k-i-1}).
\end{align*}
\end{proof}

Motivated by Lemma~\ref{L:SS:lifting D_2:simp}\ref{L:SS:lifting D_2:simp:3}, we extend the map $G$ linearly to $\A$, by setting
\begin{equation}\label{D:SS:extend:G}
 G(f\y^\ell)=G(f)\y^\ell, \qquad \mbox{for all $f\in\FF[x]$ and all $\ell\geq 0$.}
\end{equation}
Thus, we can rewrite Lemma~\ref{L:SS:lifting D_2:simp}\ref{L:SS:lifting D_2:simp:3} as 
\begin{equation}\label{L:SS:extend:G:new:c}
\s_1(\y \s_0(a \ot b))=G(a)b, \qquad \mbox{for all $a, b\in\A$.}
\end{equation}

We are now ready to define the $D_0$-operator $D_2$ in terms of $D_1$ and the homotopy $\s_1$.

\begin{lemma}\label{L:SS:lifting D_2}
There is a unique $D_0$-operator $D_2:\A\ot\R\ot\A\longrightarrow\A\ot\R\ot\A$ such that $D_2(1\ot \rr\ot 1)=\s_1\circ D_1\circ \dd_1(1\ot\rr\ot 1)$. Then $D_1\circ\dd_1=\dd_1 \circ D_2$ and
\begin{equation}\label{D:SS:lifting D_2:def:D_2}
D_2(1\ot \rr\ot 1)=G(D(x))+\s_1(D(\y)\ot x\ot 1)-\s_1\circ D_1\circ\s_0(h\ot 1).
\end{equation}
\end{lemma}
\begin{proof}
By Lemma~\ref{L:SS:lifting D_2:general}\,\ref{L:SS:lifting D_2:general:1}, there exists a unique $D_0$-operator $D_2$ defined on $\A\ot\R\ot\A$ and such that $D_2(1\ot \rr\ot 1)=\s_1\circ D_1\circ \dd_1(1\ot\rr\ot 1)$. The exact expression for $D_2(a\ot\rr\ot b)$ can be computed as in the proof of Lemma~\ref{L:SS:lifting D_1}. 

Now, using Lemma~\ref{L:SS:lifting D_2:simp} and \eqref{L:SS:extend:G:new:c}, we have
\begin{align*}
D_2(1\ot \rr\ot 1)&= \s_1(D_1(1\ot\y\ot x)) + \s_1(D_1(\y\ot x\ot 1)) -\s_1(D_1(1\ot x\ot \y)) \\ &\qquad- \s_1(D_1(x\ot\y\ot 1)) -\s_1(D_1(\s_0(h\ot 1)))\\
&= \s_1(\s_0(D(\y)\ot x))+ \s_1(1\ot\y\ot D(x)) + \s_1(\y\s_0 (D(x)\ot 1))\\ &\qquad+\s_1(D(\y)\ot x\ot 1) -\s_1(\s_0(D(x)\ot \y)) -\s_1(1\ot x\ot D(\y))
\\ &\qquad - \s_1(x\s_0(D(\y)\ot 1))- \s_1(D(x)\ot\y\ot 1) -\s_1(D_1(\s_0(h\ot 1)))
\\&= \s_1(\y\s_0 (D(x)\ot 1))+\s_1(D(\y)\ot x\ot 1) -\s_1(D_1(\s_0(h\ot 1)))
\\&= G(D(x))+\s_1(D(\y)\ot x\ot 1) -\s_1(D_1(\s_0(h\ot 1))).
\end{align*}

Finally, by Lemma~\ref{L:SS:lifting D_2:general}\,\ref{L:SS:lifting D_2:general:2}, it is enough to show that $D_1\circ\dd_1(1\ot \rr\ot 1)=\dd_1 \circ D_2(1\ot \rr\ot 1)$, so we compute, using Lemma~\ref{L:rel2} and Lemma~\ref{L:SS:lifting D_1}:
\begin{align*}
\dd_1 \circ D_2(1\ot \rr\ot 1)&=\dd_1\circ\s_1\circ D_1\circ \dd_1(1\ot\rr\ot 1)\\
&=D_1\circ \dd_1(1\ot\rr\ot 1)-\s_{0}\circ\dd_0\circ D_1\circ \dd_1(1\ot\rr\ot 1)\\
&=D_1\circ \dd_1(1\ot\rr\ot 1)-\s_{0}\circ D_0\circ \dd_0\circ \dd_1(1\ot\rr\ot 1)\\
&=D_1\circ \dd_1(1\ot\rr\ot 1),
\end{align*}
as $\dd_0\circ \dd_1=0$.
\end{proof}

\subsection{Technical lemmas}\label{SS:gb:tl}

We need to prove yet some more technical results which will allow us to simplify the computation of the Gerstenhaber bracket given in \eqref{E:SS:MSAmethod:gb1}. Although these will be particularly useful in case $\chara(\FF)=0$, most statements hold over an arbitrary field, so we include them here. 

Following \cite[Lem.\ 2.13]{BLO15ja}, it will be useful to define, for $0\neq f\in\FF[x]$, the element $\pi_f$ such that:
\begin{enumerate}
    \item $\pi_f\in\FF[x]$ is monic,
    \item $\pi_{f}= \frac{f}{\gcd (f, f')}$, up to a nonzero scalar.
\end{enumerate}

In this subsection we will mostly work over some homomorphic image of $\A$ and we will extensively use the notations $a\equiv b\modd{I}$ and $a\equiv b\modd{c}$, defined in the introduction to mean that $a-b\in I$ and $a-b\in c\A=\A c$, for a two-sided ideal $I$ and a normal element $c$, respectively. We remark that the monoid of normal elements of $\A$ was described in \cite[Thm.\ 7.2]{BLO15tams} and, in particular, any product of factors of $h$ is normal in $\A$.

\begin{lemma}\label{L:SS:gb:tl:Dmod}
Let $D\in\der(\A)$, $a\in\A$ and $k\geq 0$. The following hold:
\begin{enumerate}[label=\textup{(\alph*)}]
\item $D(h)\in h\A$ and $D(x)\in\pi_h \A$;\label{L:SS:gb:tl:Dmod:1}
\item $D(a^k)\equiv ka^{k-1}D(a)\modd{h}$;\label{L:SS:gb:tl:Dmod:2}
\item $D(\gcd(h, h'))\in \gcd(h, h')\A$.\label{L:SS:gb:tl:Dmod:3}
\end{enumerate}
\end{lemma}
\begin{proof}
The defining relation for $\A$ implies that
\begin{equation*}
D(h)=-\gb{D(x), \y} -\gb{x, D(\y)}\in\gb{\A, \A}\subseteq h\A.
\end{equation*}
So $D(h\A)\subseteq h\A$ and $D$ induces a derivation $\barr{D}:\A/h\A\longrightarrow \A/h\A$ with $\barr{D}(a+h\A)=D(a)+h\A$. Since $\A/h\A$ is commutative, we have
\begin{equation*}
D(a^k)+h\A=\barr{D}\left((a+h\A)^k\right)=ka^{k-1}D(a)+h\A, 
\end{equation*}
which proves \ref{L:SS:gb:tl:Dmod:2}.

In particular, $0\equiv D(h)\equiv h'D(x)\modd{h}$, and it follows that $h'D(x)\in h\A$. Since for any $f\in \FF[x]$ we have $h$ divides $h'f$ if and only if $\pi_h$ divides $f$, we conclude that $D(x)\in\pi_h\A$, finishing the proof of \ref{L:SS:gb:tl:Dmod:1}.

Let $g=\gcd(h, h')$. Up to a nonzero scalar, $h=\pi_h g$. Write $D(x)=\pi_h b$ for some $b\in\A$. By \ref{L:SS:gb:tl:Dmod:2},
\begin{equation*}
D(g)\in g'\pi_h b + h\A\subseteq  g'\pi_h\A + h\A.
\end{equation*}
As $h'=\pi_h g' + \pi_h' g$ and $g$ divides $h'$, we deduce that $g$ divides $\pi_h g'$, so $D(g)\in g\A + h\A=g\A$.
\end{proof}

\begin{lemma}\label{L:SS:gb:tl:s_mod}
Let $\nu$ be a divisor of $h$, $D\in\der(\A)$, $\chi\in \Hom_{\A^e}(\A\ot\R\ot\A, \A)$ and $f\in\FF[x]$. The following hold:
\begin{enumerate}[label=\textup{(\alph*)}]
\item $\s_1(\nu\A\ot\V\ot\A+\A\ot\V\ot\nu\A)\subseteq \nu\A\ot\R\ot\A+\A\ot\R\ot\nu\A$. \label{L:SS:gb:tl:s_mod:1}
\item $\chi(\nu\A\ot\R\ot\A+\A\ot\R\ot\nu\A)\subseteq \nu\A$. \label{L:SS:gb:tl:s_mod:2}
\item $\chi\circ G(f)\equiv f'\chi(1\ot\rr\ot 1)\modd{h}$; in particular, $\chi\circ G(h\A)\subseteq \gcd(h, h')\A$. \label{L:SS:gb:tl:s_mod:3}
\item If $\chara(\FF)\neq 2$, then $\chi\circ\s_1\circ D_1\circ\s_0(f\ot 1)\in \pi_h f''\A+h\A$; in particular, $\chi\circ\s_1\circ D_1\circ\s_0(h\ot 1)\in \gcd(h, h')\A$. \label{L:SS:gb:tl:s_mod:4}
\item $\chi\circ \s_1(\y^\ell\ot x\ot 1)\equiv \ell\chi(1\ot\rr\ot 1)\y^{\ell-1}\modd{\gcd(h, h')}$, for all $\ell\geq 0$. \label{L:SS:gb:tl:s_mod:5}
\end{enumerate}
\end{lemma}
\begin{proof}
The claim in \ref{L:SS:gb:tl:s_mod:1} is clear because $\nu$ is normal, $\s_1(\nu\A\ot\V\ot\A)=\nu\s_1(\A\ot\V\ot\A)\subseteq \nu\A\ot\R\ot\A$, by Lemma~\ref{L:SS:lifting D_2:simp}, and $\s_1$ is a right $\A$-module map. Claim \ref{L:SS:gb:tl:s_mod:2} is proved similarly.

Take $f=x^k$, with $k\geq 0$. Then
\begin{align*}
\chi\circ G(x^k)&= \sum_{i=0}^{k-1} x^i \chi(1\ot \rr\ot 1)x^{k-i-1}\\
&\equiv \sum_{i=0}^{k-1} x^{k-1} \chi(1\ot \rr\ot 1) \\
&\equiv  kx^{k-1} \chi(1\ot \rr\ot 1) \modd{h},
\end{align*}
establishing the first claim in \ref{L:SS:gb:tl:s_mod:3}.
Thus, for all $\ell\geq 0$,
\begin{align*}
\chi\circ G(hf\y^\ell)&= \chi(G(hf))\y^\ell  \in (h'f + hf')\chi(1\ot\rr\ot 1)\y^\ell + h\A\subseteq \gcd(h, h')\A,
\end{align*}
proving that $\chi\circ G(h\A)\subseteq \gcd(h, h')\A$. 

For \ref{L:SS:gb:tl:s_mod:4}, consider $f=x^k$, with $k\geq 0$. By Lemma~\ref{L:SS:gb:tl:Dmod}, there is $a\in\A$ such that $D(x)=\pi_h a$ and $D(x^i)-ix^{i-1}D(x)\in h\A$, for all $i\geq 0$. Set $\theta_i=D(x^i)-ix^{i-1}D(x)$. By Lemma~\ref{L:SS:lifting D_2:simp} we have:
\begin{align*}
\chi\circ\s_1\circ D_1\circ\s_0(x^k\ot 1)&= \sum_{i=1}^{k-1}\chi\circ\s_1(D(x^i)\ot x\ot x^{k-i-1})\\
&= \sum_{i=1}^{k-1}\chi\circ\s_1((ix^{i-1}D(x)+\theta_i)\ot x\ot x^{k-i-1}).
\end{align*}
By  \ref{L:SS:gb:tl:s_mod:1} and \ref{L:SS:gb:tl:s_mod:2}, $\sum_{i=1}^{k-1}\chi\circ\s_1(\theta_i\ot x\ot x^{k-i-1})\in h\A$. Thus, working modulo $h\A$ and using the commutativity of $\A/h\A$ and the hypothesis that $\chara(\FF)\neq 2$, we obtain
\begin{align*}
\chi\circ\s_1\circ D_1\circ\s_0(x^k\ot 1)
&\equiv  \sum_{i=1}^{k-1}\chi\circ\s_1(ix^{i-1}\pi_h a\ot x\ot x^{k-i-1})\\
&\equiv  \sum_{i=1}^{k-1}ix^{i-1}\pi_h \chi(\s_1(a\ot x\ot 1))x^{k-i-1}\\
&\equiv  {k\choose 2} x^{k-2}\pi_h \chi(\s_1(a\ot x\ot 1))\\
&\equiv  \left(x^{k}\right)''\pi_h \frac{1}{2} \chi(\s_1(a\ot x\ot 1)) \modd{h},
\end{align*}
so indeed $\chi\circ\s_1\circ D_1\circ\s_0(f\ot 1)\in f''\pi_h \A +h\A$. In particular,
\begin{equation*}
\chi\circ\s_1\circ D_1\circ\s_0(h\ot 1)\in h''\pi_h \A +h\A\subseteq \gcd(h, h')\A, 
\end{equation*}
because $\gcd(h, h')$ divides $h''\pi_h$.

Lastly, we prove \ref{L:SS:gb:tl:s_mod:5} by induction on $\ell\geq 0$. As 
$\chi\circ \s_1(1\ot x\ot 1)=0$, the base step is established and we assume that
\begin{equation*}
\chi\circ \s_1(\y^\ell\ot x\ot 1)\equiv \ell\chi(1\ot\rr\ot 1)\y^{\ell-1}\modd{\gcd(h, h')} 
\end{equation*}
holds for some $\ell\geq 0$. Then, by the definition of $\s_1$, the commutativity of $\A/\gcd(h, h')\A$ and part \ref{L:SS:gb:tl:s_mod:3} above, as $\delta^j(x)\in h\A$ for all positive $j$,
\begin{align*}
\chi\circ \s_1(\y^{\ell+1}\ot x\ot 1)&=
\y \chi(\s_1(\y^{\ell}\ot x\ot 1))+\sum_{j=0}^\ell {\ell\choose j} \chi\circ G\left(\delta^j(x)\right) \y^{\ell-j}\\
&\equiv
\ell\chi(1\ot\rr\ot 1)\y^{\ell} + \chi\circ G(x) \y^{\ell} \\
&\equiv
\ell\chi(1\ot\rr\ot 1)\y^{\ell} + \chi(1\ot\rr\ot 1) \y^{\ell} \modd{\gcd(h, h')}.
\end{align*}
\end{proof}

\begin{lemma}\label{L:SS:gb:tl:ynyhatn}
Let $\chi\in \Hom_{\A^e}(\A\ot\R\ot\A, \A)$, $f\in\FF[x]$ and $k\geq 0$. Then:
\begin{enumerate}[label=\textup{(\alph*)}]
\item $\pi_h h^{k-1}\gb{y^{k+1}, h}\equiv (k+1)\pi_h h'h^{k-1}y^k +{k+1\choose 2}\pi_h h'' h^{k-1} y^{k-1} \modd{h}$. (Notice that in case $k=0$ the above expression still makes sense, as $\frac{\pi_h h'}{h}=\frac{h'}{\gcd(h, h')}\in\FF[x]$.) \label{L:SS:gb:tl:ynyhatn:1}
\item $\y^k\equiv h^k y^k \modd{\gcd(h, h')}$. \label{L:SS:gb:tl:ynyhatn:2}
\item $\chi\circ G(fh^k y^k)\equiv f'\chi(1\ot\rr\ot 1)\y^k-{k+1\choose 2}fh''\chi(1\ot\rr\ot 1)\y^{k-1} \modd{\gcd(h, h')}$. \label{L:SS:gb:tl:ynyhatn:3}
\end{enumerate} 
\end{lemma}

\begin{proof}
Working modulo $h\A$, we deduce \ref{L:SS:gb:tl:ynyhatn:1}:
\begin{align*}
\pi_h h^{k-1}\gb{y^{k+1}, h} 
&=\sum_{j=1}^{k+1}{k+1\choose j}\pi_h h^{(j)}h^{k-1}y^{k+1-j}\\
&\equiv (k+1)\pi_h h'h^{k-1}y^{k}
+{k+1\choose 2}\pi_h h''h^{k-1}y^{k-1} \modd{h}.
\end{align*}
In particular, multiplying both sides of \ref{L:SS:gb:tl:ynyhatn:1} by $\gcd(h, h')=h/\pi_h$ we obtain
\begin{align}\label{E:SS:gb:tl:ynyhatn:b}
\begin{split}
h^{k}\gb{y^{k+1}, h} &\equiv (k+1)h'h^{k}y^k +{k+1\choose 2} h'' h^{k} y^{k-1} \\&\equiv (k+1)h'h^{k}y^k \modd{h},
\end{split}
\end{align}
and it follows that $h^{k}\gb{y^{k+1}, h}\in \gcd(h, h')\A$.

We are now ready to prove \ref{L:SS:gb:tl:ynyhatn:2} by induction on $k\geq 0$, the base case being trivial. Supposing that \ref{L:SS:gb:tl:ynyhatn:2} holds for a certain $k\geq 0$, we get
\begin{align*}
\y^{k+1}&\equiv h^k y^{k+1}h=h^{k+1} y^{k+1} + h^k \gb{y^{k+1},h}\equiv 
h^{k+1} y^{k+1} \modd{\gcd(h, h')}.
\end{align*}

We also prove \ref{L:SS:gb:tl:ynyhatn:3} by induction on $k\geq 0$. The case $k=0$ is immediate from Lemma~\ref{L:SS:gb:tl:s_mod}\ref{L:SS:gb:tl:s_mod:3}. For the inductive step, assume the congruence holds for $k\geq 0$. By \eqref{E:SS:gb:tl:ynyhatn:b} we have
\begin{align*}
h^{k+1} y^{k+1}&=h^ky^{k+1}h - h^k \gb{y^{k+1}, h}\equiv
h^ky^{k}\y -(k+1)h'h^{k}y^k \modd{h}.
\end{align*}
By Lemma~\ref{L:SS:gb:tl:s_mod}\,\ref{L:SS:gb:tl:s_mod:3},
\begin{align*}
\chi\circ G(fh^{k+1} y^{k+1})&\equiv \chi\circ G(fh^ky^{k})\y -(k+1)\chi\circ G(fh'h^{k}y^k)\\
&\equiv f'\chi(1\ot\rr\ot 1)\y^{k+1}-{k+1\choose 2}fh''\chi(1\ot\rr\ot 1)\y^{k}\\
&\quad -(k+1)(fh')'\chi(1\ot\rr\ot 1)\y^k\\ &\quad +(k+1){k+1\choose 2}fh'h''\chi(1\ot\rr\ot 1)\y^{k-1} \\
&\equiv f'\chi(1\ot\rr\ot 1)\y^{k+1}-{k+1\choose 2}fh''\chi(1\ot\rr\ot 1)\y^{k}\\
&\quad -(k+1)fh''\chi(1\ot\rr\ot 1)\y^k\\
&\equiv f'\chi(1\ot\rr\ot 1)\y^{k+1}-{k+2\choose 2}fh''\chi(1\ot\rr\ot 1)\y^{k}\\ &\omit $\hfill \modd{\gcd(h, h')}.$
\end{align*}
\end{proof}

\section{The Gerstenhaber bracket}\label{S:Gerst:char0}

In this section we determine the structure of $\hoch^2(\A)$ as a module over the Lie algebra $\hoch^1(\A)$ under the Gerstenhaber bracket, always under the assumption that $\chara(\FF)=0$. We will prove some of the main results of this article. In the first subsection we will describe two different subspaces of the space of linear derivations of our algebra, that will act on $\hoch^2(\A)$ in a very different way. Next we will describe the action of the classes of these derivations on $\hoch^2(\A)$. Then we achieve our goal of giving an explicit description of $\hoch^2(\A)$ as $\hoch^1(\A)$-Lie module. We finish the section by relating this action of $\hoch^1(\A)$ on $\hoch^2(\A)$ with the representation theory of the Virasoro algebra, and then by discussing several special cases.

\subsection{The Lie algebra structure of $\hoch^1(\A)$}\label{SS:Gerst:char0:LieHH1}
The Lie algebra structure of $\hoch^1(\A)$ in case $\chara(\FF)=0$ is described explicitly in \cite[Sec.\ 5]{BLO15ja} and we briefly collect the results we need below. 

There are two types of derivations of $\A$, which together describe $\der(\A)$ and $\hoch^1(\A)$:
\begin{itemize}
\item For any $g\in\FF[x]$, let $D_g$ be the derivation of $\A$ such that $D_g(x)=0$ and $D_g(\y)=g$. Then, $\{ D_g \mid g\in\FF[x] \}$ is an abelian Lie subalgebra of $\der(\A)$ and $D_g\in\inder(\A)$ if and only if $g\in h\FF[x]$.
\item Viewing, as usual, $\A=\A_h\subseteq \A_1$ with $\y=yh$, define the elements $a_n=\pi_h h^{n-1}y^n\in \{ u\in\A_1 \mid \gb{u, \A}\subseteq \A \}$ (the normalizer of $\A$ in $\A_1$), for all $n\geq 1$. It will also be convenient to consider the element $a_0=\pi_h/h=\frac{1}{\gcd(h, h')}$ in the localization of $\A_1$ at the Ore set formed by the powers of $h$. Then, $\ad_{ga_n}\in\der(\A)$ for all for all $n\geq 0$ and $g\in\FF[x]$. Moreover, $\ad_{ga_n}\in\inder(\A)$ if and only if $g\in\gcd(h, h')\FF[x]$. 
\end{itemize}

Next, we recall the definition in \cite[Sec.\ 4.3]{BLO15ja} of the linear endomorphism $\delta_0:\FF[x]\longrightarrow\FF[x]$ given by 
\begin{equation}\label{D:Gerst:char0:defdelta0}
\delta_0(g)=\delta(ga_0)=(g\pi_h h^{-1})'h=(g\pi_h)' - g\frac{\pi_h h'}{h}. 
\end{equation} 
By \cite[Lem.\ 4.14]{BLO15ja}, $\ad_{ga_0}=-D_{\delta_0(g)}$.

For notational simplicity, by \cite[Thm.\ 8.2]{BLO15tams}, we can assume that $h$ is monic, say $h=\pr_1^{\alpha_1}\cdots \pr_t^{\alpha_t}$, where $\pr_1, \ldots, \pr_t$ are the distinct monic prime factors of $h$, with multiplicities $\alpha_1, \ldots, \alpha_t$. Up to changing the order of the factors, we can further assume that there is $0\leq k\leq t$ such that $\alpha_1, \ldots, \alpha_k\geq 2$ and $\alpha_{k+1}= \cdots = \alpha_t=1$.
Moreover, if $k=0$ then $\gcd(h, h')=1$ and in this case $\hoch^2(\A)=0$, so there is nothing to prove.

We have the following result (see also {\cite[Thm.\ 5.1, Prop.\ 5.9]{BLO15ja}}).

\begin{thm} \label{T:Gerst:char0:recallBLO3}  
Assume  $\chara(\FF) = 0$.   Then there is a decomposition
$\hoch^1(\A) = \mathsf{Z}(\hoch^1(\A))  \oplus [\hoch^1(\A),\hoch^1(\A)]$. 
Moreover, using the above notations,
there are isomorphisms of Lie algebras:
\begin{enumerate}[label=\textup{(\alph*)}]
\item $\mathcal N=\spann_\FF\{\ad_{g a_n} \mid  g \in \pr_1 \cdots \pr_k\FF[x], \ n \geq 0\}$ is the unique maximal  nilpotent ideal  of 
$ [\hoch^1(\A), \hoch^1(\A)]$.\label{T:Gerst:char0:recallBLO3:1}
\item $\mathsf{Z}(\hoch^1(\A)) \cong \left\{ D_{g} \mid g\in\gcd(h, h')\FF[x], \   \degg g <\degg h \right\}$.\label{T:Gerst:char0:recallBLO3:2}
\item $[\hoch^1 (\A), \hoch^1 (\A)]  =  \spann_\FF\{\ad_{ga_n}  \mid  g \in\FF[x],  \degg g <\degg \gcd(h, h'),  n\geq 0\}$.\label{T:Gerst:char0:recallBLO3:3}
\item $[\hoch^1 (\A), \hoch^1 (\A)]/\mathcal N \cong \W_1\oplus\cdots\oplus\W_k$, where $\W_i=\left(\FF[x]/\pr_{i}\FF[x]\right) \otimes \W$ is a field extension of the Witt algebra.\label{T:Gerst:char0:recallBLO3:4}
\end{enumerate} 
\end{thm}

\subsection{Formulas for the Gerstenhaber bracket $\gb{\hoch^1(\A), \hoch^2(\A)}$}\label{SS:Gerst:char0:formula}

Recall that by Corollary~\ref{C:res:HH2inchar0}, $\hoch^2(\A) \cong \A/\gcd(h, h')\A$ can be identified with the polynomial ring $\D[\y]$, where $\D=\left(\FF[x]/\gcd(h, h')\FF[x]\right)$. We will use \eqref{E:SS:MSAmethod:gb1} and also the identification introduced there 
between $\A/\gcd(h, h')\A$ and $\Hom_{\A^e}(\A\ot\R\ot\A, \A)/\im\, \dd^*_1$, which associates the element $a\in\A$ with the map $\chi_a\in \Hom_{\A^e}(\A\ot\R\ot\A, \A)$ defined by $\chi_a(1\ot\rr\ot 1)=a$, and similarly for the corresponding homomorphic images.

Now, Lemma~\ref{L:SS:gb:tl:s_mod}\ref{L:SS:gb:tl:s_mod:4} implies that for all $a\in\A$, the image of $\chi_a\circ\s_1\circ D_1\circ\s_0(h\ot 1)$ in $\hoch^2(\A)$ is zero. Thus we have, using Lemma~\ref{L:SS:lifting D_2},
\begin{equation}\label{E:SS:MSAmethod:gb2}
\gb{D, a}=D(a)-\chi_a(G(D(x))) -\chi_a(\s_1(D(\y)\ot x\ot 1)), 
\end{equation}
for all $a\in \A$ and $D\in\der(\A)$. Moreover, by Lemma~\ref{L:SS:gb:tl:Dmod}\ref{L:SS:gb:tl:Dmod:3}, the image of $D(a)$ in $\hoch^2(\A)$ depends only on the class $a+\gcd(h, h')\A$ and similarly, $\chi_a(G(D(x)))$ and $\chi_a(\s_1(D(\y)\ot x\ot 1))$ depend only on the classes $D(x)+h\A$ and $D(\y)+\gcd(h, h')\A$, respectively, by Lemma~\ref{L:SS:gb:tl:s_mod}.

We will first consider the derivations of the form $D_g$, for $g\in\FF[x]$. Fix $g$ and let $D=D_g$. Take $a=p(x)\y^k$ for some $p(x)\in\FF[x]$ and $k\geq 0$. Then $D(x)=0=\s_1(D(\y)\ot x\ot 1)$ and by Lemma~\ref{L:SS:gb:tl:Dmod}, $D(p(x)\y^k)=p(x)D(\y^k)\equiv kp(x)\y^{k-1}g\equiv kgp(x)\y^{k-1}\modd{h}$. Thus, $\gb{D_g, p(x)\y^k}\equiv kgp(x)\y^{k-1} \modd{\gcd(h, h')}$. So,
\begin{equation}\label{E:SS:Gerst:char0:gbDg} 
\gb{D_g, -}=g\frac{d}{d\y} \qquad \mbox{on $\D[\y]$.}
\end{equation}
In particular, $\gb{\mathsf{Z}(\hoch^1(\A)), \hoch^2(\A)}=0$, by Theorem~\ref{T:Gerst:char0:recallBLO3}\ref{T:Gerst:char0:recallBLO3:2}.

Now we can turn our attention to the derivations of the  form $\ad_{ga_n}$, with $g\in\FF[x]$ and $n\geq 0$.

\begin{lemma}\label{L:SS:Gerst:char0:Damodgcd} 
Let $D=\ad_{ga_n}$ and $a=p(x)\y^k\in\A$, as above. Then:
\begin{enumerate}[label=\textup{(\alph*)}]
\item $D(x)=n\pi_h g h^{n-1}y^{n-1}\equiv n\pi_h g \y^{n-1} \modd{\gcd(h, h')}$.\label{L:SS:Gerst:char0:Damodgcd:1}
\item $D(\y)\equiv -\delta_0(g)\y^n \modd{\gcd(h, h')}$.
\item $D(a)\equiv\left( n\pi_h gp'(x)-kp(x)\delta_0(g) \right)\y^{n+k-1} \modd{\gcd(h, h')}$.
\end{enumerate}
\end{lemma}

\begin{proof}
We have 
\begin{equation*}
D(x)=\gb{\pi_h gh^{n-1}y^n, x}=n\pi_h gh^{n-1}y^{n-1}\equiv n\pi_h g\y^{n-1} \modd{\gcd(h, h')}, 
\end{equation*}
where the last congruence comes from Lemma~\ref{L:SS:gb:tl:ynyhatn}\ref{L:SS:gb:tl:ynyhatn:2}. Also, 
\begin{align*}
D(\y)&= \gb{\pi_h gh^{n-1}y^n, \y}=\pi_h gh^{n-1}y^{n+1}h - y\pi_h gh^{n}y^n \\
&=\pi_h gh^{n}y^{n+1} + \pi_h gh^{n-1}\gb{y^{n+1},h} - \pi_h gh^{n}y^{n+1} - \gb{y,\pi_h gh^{n}}y^n\\
&\equiv  (n+1)\pi_h h' g h^{n-1}y^n +{n+1\choose 2}\pi_h g h'' h^{n-1} y^{n-1}   - \left(\pi_h gh^{n}\right)' y^n \modd{h}\\
&\equiv  (n+1)\pi_h h' g h^{n-1}y^n  - \left(\pi_h gh^{n}\right)' y^n \modd{\gcd(h, h')}\\
&\equiv  (n+1)\pi_h h' g h^{n-1}y^n  - n\pi_h gh' h^{n-1} y^n - \left(\pi_h g\right)' h^{n}y^n\modd{\gcd(h, h')}\\
&\equiv  \pi_h h' g h^{n-1}y^n  - \left(\pi_h g\right)' h^{n}y^n\modd{\gcd(h, h')}\\
&\equiv  \left(\frac{\pi_h h' g}{h}  - \left(\pi_h g\right)'\right) h^{n}y^n\modd{\gcd(h, h')}\\
&\equiv  -\delta_0(g)\y^n\modd{\gcd(h, h')},
\end{align*}
using Lemma~\ref{L:SS:gb:tl:ynyhatn}\ref{L:SS:gb:tl:ynyhatn:1} and \ref{L:SS:gb:tl:ynyhatn:2}, the fact that $\gcd(h, h')$ divides $h''\pi_h$ and \eqref{D:Gerst:char0:defdelta0}.

Finally, using Lemma~\ref{L:SS:gb:tl:Dmod}\,\ref{L:SS:gb:tl:Dmod:2},
\begin{align*}
D(a) &\equiv  D(p(x))\y^k+ p(x)D(\y^k)\\
&\equiv  p'(x)D(x)\y^k+ kp(x)D(\y)\y^{k-1}\\
&\equiv  \left(n\pi_h g p'(x)- kp(x)\delta_0(g)\right)\y^{n+k-1} \modd{\gcd(h, h')}.
\end{align*}
\end{proof}

Hence, for $D=\ad_{ga_n}$ and $a=p(x)\y^k\in\A$, we can now compute $\gb{D, a}$ as an element of $\D[\y]$, using \eqref{E:SS:MSAmethod:gb2}, Lemma~\ref{L:SS:gb:tl:s_mod}\ref{L:SS:gb:tl:s_mod:5}, Lemma~\ref{L:SS:gb:tl:ynyhatn}\ref{L:SS:gb:tl:ynyhatn:3} and recalling that $\gcd(h, h')$ divides $h''\pi_h$:
\begin{align*}
D(a)&\equiv\left( n\pi_h gp'(x)-kp(x)\delta_0(g) \right)\y^{n+k-1} \modd{\gcd(h, h')}
\end{align*}
\begin{align*}
\chi_a(G(D(x)))&=\chi_a(G(n\pi_h g h^{n-1}y^{n-1}))\\ 
&\equiv n (\pi_h g)'p(x)\y^{n+k-1} -n{n\choose 2}\pi_h g h''p(x)\y^{n+k-2}\\
&\equiv n (\pi_h g)'p(x)\y^{n+k-1}  \modd{\gcd(h, h')}
\end{align*}
\begin{align*}
\chi_a(\s_1(D(\y)\ot x\ot 1))&\equiv -\delta_0(g)\chi_a(\s_1(\y^n\ot x\ot 1))\\
&\equiv -n\delta_0(g)p(x)\y^{n+k-1}\modd{\gcd(h, h')}.
\end{align*}
It thus follows that, working in $\hoch^2(\A)=\A/\gcd(h, h')\A$ and recalling \eqref{D:Gerst:char0:defdelta0}:
\begin{align*}
\gb{D, a}&\equiv n\left( \pi_h g p'(x)- (\pi_h g)'p(x) \right)\y^{n+k-1} +
(n-k)p(x)\delta_0(g)\y^{n+k-1}\\
&\equiv \left(n\pi_h g p'(x) -ng\frac{\pi_h h'}{h}p(x)
-k\delta_0(g)p(x)\right) \y^{n+k-1} \modd{\gcd(h, h')}.
\end{align*}
Therefore, we have proved the main result of this subsection.

\begin{thm}\label{T:Gerst:char0:formula}
Assume that $\chara(\FF)=0$. The Lie action of $\hoch^1(\A)$ on $\hoch^2(\A)$ under the Gerstenhaber bracket is given by the following formulas:
\begin{align}\label{E:Gerst:char0:formula:actZ}
\gb{\mathsf{Z}(\hoch^1(\A)), \hoch^2(\A)}=0,
\end{align}
\begin{align}\label{E:Gerst:char0:formula:actadgan}
\gb{\ad_{ga_n}, -}&= n\pi_h g\y^{n-1}\frac{d}{dx}-\delta_0(g)\y^n\frac{d}{d\y}-ng\frac{\pi_h h'}{h}\y^{n-1}1_{\D[\y]},
\end{align}
for all $g\in\FF[x]$ and $n\geq 0$, where $a_n=\pi_h h^{n-1}y^n$.
\end{thm}

\subsection{The structure of $\hoch^2(\A)$ as a Lie module over $\hoch^1(\A)$}\label{SS:Gerst:char0:struct:gc}

Recall that $h=\pr_1^{\alpha_1}\cdots \pr_t^{\alpha_t}$, where $\pr_1, \ldots, \pr_t$ are the prime factors of $h$, ordered so that $\alpha_1, \ldots, \alpha_k\geq 2$ and $\alpha_{k+1}= \cdots = \alpha_t=1$ for $0\leq k\leq t$, as in Theorem~\ref{T:Gerst:char0:recallBLO3}. If $k=0$, then $\gcd(h, h')=1$ and in this case $\hoch^2(\A)=0$. Thus, we suppose throughout this subsection that $k\geq 1$. Then, 
\begin{equation*}
\pi_h= \pr_1\cdots \pr_t, \quad \gcd(h, h')=h/\pi_h=\pr_1^{\alpha_1-1}\cdots \pr_k^{\alpha_k-1}, \quad \pi_{(h/\pi_h)}=\pr_1\cdots \pr_k.
\end{equation*}
Let us fix $m_h=\max\{ \alpha_j-1 \mid 1\leq j\leq k\}\geq 1$.

We make the identification $\hoch^2(\A)=\D[\y]$, where $\D=\FF[x]/\gcd(h, h')\FF[x]$. Since the $\pr_i^{\alpha_i-1}$, $1\leq i\leq k$, are pairwise coprime, 
\begin{equation*}
\D\cong \FF[x]/\pr_1^{\alpha_1-1}\FF[x] \oplus \cdots \oplus \FF[x]/\pr_k^{\alpha_k-1}\FF[x]
\end{equation*}
and there exist nonzero pairwise orthogonal idempotents $e_1, \ldots, e_k\in\D$ with $e_1 + \cdots + e_k=1$, $\D=\D e_1\oplus \cdots \oplus \D e_k$ and $\D e_i\cong \FF[x]/\pr_i^{\alpha_i-1}\FF[x]$ (these isomorphisms are both as algebras and as left $\FF[x]$-modules). Define $\D_i=\D e_i$. Then $\hoch^2(\A)=\D_1[\y]\oplus \cdots\oplus\D_k[\y]$.

Let $\barr{\D}=\FF[x]/\pr_1\cdots\pr_k\FF[x]\cong \FF[x]/\pr_1\FF[x]\oplus\cdots\oplus\FF[x]/\pr_k\FF[x]$. Then, by Theorem~\ref{T:Gerst:char0:recallBLO3}\ref{T:Gerst:char0:recallBLO3:4}, we have 
\begin{equation*}
[\hoch^1 (\A), \hoch^1 (\A)]/\mathcal N \cong \barr{\D} \otimes \W\cong \W_1\oplus\cdots\oplus\W_k,
\end{equation*}
with $\W_i=\left(\FF[x]/\pr_{i}\FF[x]\right) \otimes \W$. As the notation suggests, the algebra $\barr{\D}$ is a quotient of $\D$ by the ideal $\pr_1\cdots\pr_k\D$. Let $\barr{e_1}, \ldots, \barr{e_k}\in\barr{\D}$ be the images of the idempotents $e_1, \ldots, e_k\in\D$ under the canonical epimorphism. It is straightforward to see that these are still nonzero pairwise orthogonal idempotents in $\barr{\D}$ with $\barr{e_1} + \cdots + \barr{e_k}=1$, $\barr{\D}=\barr{\D} \barr{e_1}\oplus \cdots \oplus \barr{\D} \barr{e_k}$ and $\barr{\D} \barr{e_i}\cong \FF[x]/\pr_i\FF[x]$. Denote this field $\barr{\D} \barr{e_i}=\barr{\D_i}$ by $\barr{\D}_i$. Then,
\begin{equation}\label{E:Gerst:char0:struct:ssdecomp}
[\hoch^1 (\A), \hoch^1 (\A)]/\mathcal N \cong \left(\barr{\D}_1\otimes\W\right) \oplus\cdots\oplus \left(\barr{\D}_k\otimes\W\right).
\end{equation}

For $i\geq 0$, set
\begin{equation*}
\Theta_i=\prod_{j=1}^k \pr_j^{\min\{ \alpha_j -1, i\}}.
\end{equation*}
Thus, $\Theta_0=1$, $\Theta_1=\pr_1\cdots \pr_k=\pi_{(h/\pi_h)}$ and for any $i\geq m_h$, $\Theta_i=\gcd(h, h')$. Finally, define
\begin{equation*}
P_i =\Theta_i \D[\y]\subseteq \hoch^2(\A). 
\end{equation*}

We record a few useful facts below.

\begin{lemma}\label{L:Gerst:char0:struct:thi}
For $i\geq 0$,we have:
\begin{enumerate}[label=\textup{(\alph*)}]
\item $\Theta_{i+1}=\Theta_{i}\left(\prod_{\alpha_j\geq i+2} \pr_j\right)$.\label{L:Gerst:char0:struct:thi:1}
\item $\pi_h \Theta_i'\equiv i\Theta_i\pi_h' \modd{\Theta_{i+1}\FF[x]}$.\label{L:Gerst:char0:struct:thi:2}
\item $P_i=\Theta_i \D[\y]$ is a Lie $\hoch^1(\A)$-submodule of $\hoch^2(\A)$ and there is a strictly decreasing filtration
\begin{equation}\label{E:Gerst:char0:struct:filtr}
\hoch^2(\A)=P_0\supsetneq P_1\supsetneq \cdots \supsetneq P_{m_h-1}\supsetneq P_{m_h}=0. 
\end{equation} \label{L:Gerst:char0:struct:thi:3}
\end{enumerate}
\end{lemma}

\begin{proof}
\ref{L:Gerst:char0:struct:thi:1} is clear from the definition. The identity in \ref{L:Gerst:char0:struct:thi:2} holds trivially for $i=0$ and we prove it by induction on $i\geq 0$. So assume that $\pi_h \Theta_i' = i\Theta_i\pi_h' +\Theta_{i+1}f$, for some $f\in\FF[x]$. As $\Theta_{i+1}\left(\prod_{\alpha_j\geq i+2} \pr_j\right)\in\Theta_{i+2}\FF[x]$, by \ref{L:Gerst:char0:struct:thi:1}, we have 
\begin{align*}
\pi_h \Theta_{i+1}' &=\pi_h\left(\Theta_{i}\prod_{\alpha_j\geq i+2} \pr_j\right)'=\pi_h\Theta_{i}'\left(\prod_{\alpha_j\geq i+2} \pr_j\right)
+\pi_h\Theta_{i}\left(\prod_{\alpha_j\geq i+2} \pr_j\right)'\\
&=\left( i\Theta_i\pi_h' +\Theta_{i+1}f\right)\left(\prod_{\alpha_j\geq i+2} \pr_j\right)
+\pi_h\Theta_{i}\left(\prod_{\alpha_j\geq i+2} \pr_j\right)'\\
&\equiv i\Theta_i\pi_h'\left(\prod_{\alpha_j\geq i+2} \pr_j\right) + \pi_h\Theta_{i}\left(\prod_{\alpha_j\geq i+2} \pr_j\right)' \ \modd{\Theta_{i+2}\FF[x]}\\
&= i\Theta_{i+1}\pi_h' + \Theta_{i+1}\left(\prod_{1\leq\alpha_j\leq i+1} \pr_j\right)\left(\prod_{\alpha_j\geq i+2} \pr_j\right)' \\
&= i\Theta_{i+1}\pi_h' + \Theta_{i+1}\left(\pi_h'-\left(\prod_{1\leq\alpha_j\leq i+1} \pr_j\right)'\left(\prod_{\alpha_j\geq i+2} \pr_j\right)\right)  \\
&\equiv i\Theta_{i+1}\pi_h' + \Theta_{i+1}\pi_h' \ \modd{\Theta_{i+2}\FF[x]}.
\end{align*}

The fact that \eqref{E:Gerst:char0:struct:filtr} is a decreasing filtration of vector spaces is clear because $\Theta_{i}$ divides $\Theta_{i+1}$. Since the quotient $\prod_{\alpha_j\geq i+2} \pr_j$ of these polynomials is not a unit, for $0\leq i\leq m_h-1$, by the definition of $m_h$, the filtration is strict. Thus, it remains to show that $\gb{\ad_{ga_n}, P_i}\subseteq P_i$, for all $g\in\FF[x]$ and $n, i\geq 0$. 
By \eqref{E:Gerst:char0:formula:actadgan}, given $f\in\FF[x]$ and $\ell\geq 0$:
\begin{align*}
\gb{\ad_{ga_n}, \Theta_i f\y^\ell} &= n\pi_h g\Theta_i f'\y^{n+\ell-1}+n\pi_h g\Theta_i'f\y^{n+\ell-1}\\ 
&\qquad-\ell\delta_0(g)\Theta_i f\y^{n+\ell-1}-ng\frac{\pi_h h'}{h}\Theta_i f\y^{n+\ell-1},
\end{align*}
which is in $P_i$ because $\pi_h\Theta_i'\in\Theta_i\FF[x]$.
\end{proof}

Set $S_i=P_i/P_{i+1}$, for $0\leq i\leq m_h-1$. We have seen that $S_i$ is a nonzero $\hoch^1(\A)$-module under the action induced from the Gerstenhaber bracket. Noting that $\delta_0(g)=g\delta_0(1)+g'\pi_h$ (see \cite[Lem.\ 4.14]{BLO15ja}) and $\pi_h\Theta_i\in\Theta_{i+1}\FF[x]$, we see that this action is completely described by the following computation in $S_i$:
\begin{align}\nonumber
\gb{\ad_{ga_n}, \Theta_i f\y^\ell} &\equiv 
fg\left(n\pi_h \Theta_i'-\ell \delta_0(1)\Theta_i -n\frac{\pi_h h'}{h}\Theta_i \right) \y^{n+\ell-1}, \\ \label{E:Gerst:char0:struct:actionSi}
&\equiv fg\Theta_i\left(in\pi_h' -\ell\delta_0(1)-n\frac{\pi_h h'}{h}  \right)\y^{n+\ell-1} \ \modd{P_{i+1}}.
\end{align}
In particular, $\gb{\ad_{ga_n}, S_i}=0$ if $g\in\pr_1\cdots \pr_k\FF[x]=\Theta_1\FF[x]$ because $\Theta_1\Theta_i\in\Theta_{i+1}\FF[x]$. So, $\gb{\mathcal{N}, S_i}=0$ for all $i\geq 0$, where $\mathcal{N}$ is the unique maximal  nilpotent ideal  of 
$[\hoch^1(\A), \hoch^1(\A)]$, as in Theorem~\ref{T:Gerst:char0:recallBLO3}. It follows that $S_i$ is naturally a $[\hoch^1 (\A), \hoch^1 (\A)]/\mathcal N$-module.

Note that
$S_i\cong \left(\Theta_i \D/\Theta_{i+1} \D \right)[\y]$.
Then, the definitions of $\D$, $\Theta_i$ and $m_h-1$, along with Lemma~\ref{L:Gerst:char0:struct:thi}\ref{L:Gerst:char0:struct:thi:1} imply that there is a natural isomorphism of vector spaces induced by the natural map $\D\twoheadrightarrow \Theta_i \D/\Theta_{i+1}\D$:
\begin{equation}\label{E:Gerst:char0:struct:Si}
S_i\cong \frac{\D}{\left(\prod_{\alpha_j\geq i+2} \pr_j\right) \D}[\y] \cong \bigoplus_{\alpha_j\geq i+2} \barr{\D}_j[\y],
\quad \mbox{for all $0\leq i\leq m_h-1$.}
\end{equation}
By the above isomorphisms, the element $\Theta_i f\y^\ell+\Theta_{i+1} \D[\y]\in S_i$ is identified with the element $\sum_{\alpha_j\geq i+2}
f\barr{e_j} \y^\ell\in \bigoplus_{\alpha_j\geq i+2} \barr{\D}_j[\y]$.

Our next step is to describe the Lie algebra isomorphism \eqref{E:Gerst:char0:struct:ssdecomp}. We will need the following.

\begin{lemma}\label{L:Gerst:char0:struct:nu}
There is an element $\nu\in\FF[x]$, determining a unique class modulo $\Theta_1\FF[x]$, such that $\nu\delta_0(1)\equiv 1 \modd{\Theta_1\FF[x]}$. For such an element, the following hold:
\begin{enumerate}[label=\textup{(\alph*)}]
\item $\nu\pi_h'-1\equiv \nu\frac{\pi_h h'}{h} \modd{\Theta_1\FF[x]}$;\label{L:Gerst:char0:struct:nu:1}
\item $\nu\pi_h'\equiv \frac{1}{1-\alpha_j} \modd{\pr_j\FF[x]}$, for all $1\leq j\leq k$.\label{L:Gerst:char0:struct:nu:2}
\end{enumerate}
\end{lemma}
\begin{proof}
We have $\pi_h'=\sum_{i=1}^t \pr_1 \cdots \widehat \pr_i \cdots \pr_t \pr_i'$ and $\frac{\pi_h h'}{h}=\sum_{i=1}^t \alpha_i\pr_1 \cdots \widehat \pr_i \cdots \pr_t \pr_i'$, so in particular, $\delta_0(1)=\pi_h'-\frac{\pi_h h'}{h}=\pr_{k+1}\cdots \pr_t\sum_{i=1}^k  (1-\alpha_i)  \pr_1 \cdots \widehat \pr_i \cdots \pr_k \pr_i'$ and $\gcd(\delta_0(1), \Theta_1) =1$. This shows the existence of $\nu$ with $\nu\delta_0(1)\equiv 1 \modd{\Theta_1\FF[x]}$ and also proves \ref{L:Gerst:char0:struct:nu:1}. 

Fix $1\leq j\leq k$. Then $\pi_h'\equiv \pr_1 \cdots \widehat \pr_j \cdots \pr_t \pr_j' \modd{\pr_j\FF[x]}$ and $\frac{\pi_h h'}{h}\equiv \alpha_j\pr_1 \cdots \widehat \pr_j \cdots \pr_t \pr_j' \equiv \alpha_j \pi_h' \modd{\pr_j\FF[x]}$. But, by~\ref{L:Gerst:char0:struct:nu:1}, we also have 
$\nu\pi_h'- \nu\frac{\pi_h h'}{h}\equiv 1 \modd{\pr_j\FF[x]}$, so $(1-\alpha_j)\nu\pi_h'\equiv 1 \modd{\pr_j\FF[x]}$ and~\ref{L:Gerst:char0:struct:nu:2} follows since $\alpha_j\geq 2$.
\end{proof}

Based on the proof of \cite[Lem.\ 5.19]{BLO15ja} and the definition of $\barr{\D}_q$, we can deduce that under the isomorphism \eqref{E:Gerst:char0:struct:ssdecomp}, the element $g\barr{e_q}\otimes w_m\in \barr{\D}_q\otimes\W$ is mapped to $-\ad_{ge_q\nu a_{m+1}}+\mathcal{N}\in [\hoch^1 (\A), \hoch^1 (\A)]/\mathcal N$, for $1\leq q\leq k$, $g\in\FF[x]$ and $m\geq -1$, where $\nu$ is as in Lemma~\ref{L:Gerst:char0:struct:nu}. Using these identifications and those in \eqref{E:Gerst:char0:struct:Si}, we have:
\begin{align*}
&\left(g\barr{e_q}  \otimes w_m\right) . \left(\sum_{\alpha_j\geq i+2} f\barr{e_j} \y^\ell\right) =
-\gb{\ad_{ge_q\nu a_{m+1}}, \Theta_i f\y^\ell}\\
&\quad = \Theta_i fg\barr{e_q}\left(-i(m+1)\nu\pi_h' +\ell\nu\delta_0(1)+(m+1)\nu\frac{\pi_h h'}{h}  \right)\y^{m+\ell}
\ \modd{P_{i+1}}\\
&\quad = \Theta_i fg\barr{e_q}\left((1-i)(m+1)\nu\pi_h' +\ell-(m+1)  \right)\y^{m+\ell}
\ \modd{P_{i+1}}\\
&\quad = \sum_{\alpha_j\geq i+2}  fg\barr{e_j}\barr{e_q}\left(\ell-(m+1)\left(1-(1-i)\nu\pi_h'\right)  \right)\y^{m+\ell}\\
&\quad =
\begin{cases}
fg\barr{e_q}\left(\ell-(m+1)\left(1-(1-i)\nu\pi_h'\right)  \right)\y^{m+\ell} & \mbox{if $\alpha_q\geq i+2$}\\
0 & \mbox{if $\alpha_q\leq i+1$,}
\end{cases}
\end{align*}
by \eqref{E:Gerst:char0:struct:actionSi} and Lemma~\ref{L:Gerst:char0:struct:nu}, as $\Theta_{i+1}$ divides $\Theta_1\Theta_i$. Moreover, we can use Lemma~\ref{L:Gerst:char0:struct:nu}\ref{L:Gerst:char0:struct:nu:2}   since $\pr_q\barr{e_q}=0$ in $\barr{D}_q$, yielding:
\begin{align*}
\left(g\barr{e_q}  \otimes w_m\right) . \left(\sum_{\alpha_j\geq i+2} f\barr{e_j} \y^\ell\right) =
\begin{cases}
fg\barr{e_q}\left(\ell-(m+1)\frac{\alpha_q-i}{\alpha_q-1}  \right)\y^{m+\ell} & \mbox{if $\alpha_q\geq i+2$;}\\
0 & \mbox{if $\alpha_q\leq i+1$.}
\end{cases}
\end{align*}

The above shows that $\barr{\D}_q\otimes\W$ acts trivially on $\barr{\D}_j[\y]\subseteq S_i$ except if $j=q$ and $\alpha_q\geq i+2$. In the latter case, the action of $\barr{\D}_q\otimes\W$ on $\barr{\D}_q[\y]$ is given by
\begin{equation}\label{E:Gerst:char0:struct:Si:ss:action}
\left(g\barr{e_q}  \otimes w_m\right) . \left(f\barr{e_q} \y^\ell\right) = fg\barr{e_q}\left(\ell-(m+1)\frac{\alpha_q-i}{\alpha_q-1}  \right)\y^{m+\ell}.
\end{equation}
In particular, each $\barr{\D}_j[\y]\subseteq S_i$ in the decomposition \eqref{E:Gerst:char0:struct:Si} is an $\hoch^1(\A)$-submodule of $S_i$.

Notice that in \eqref{E:Gerst:char0:struct:Si:ss:action}, the elements $f\barr{e_q}$ and $g\barr{e_q}$ are scalars in the field extension $\barr{\D}_q\cong\FF[x]/\pr_q\FF[x]$ of $\FF$ and the action \eqref{E:Gerst:char0:struct:Si:ss:action} is $\barr{\D}_q$-linear. This motivates the following definition. Fix a scalar $\mu\in\FF$ and let $V_\mu=\FF[\y]$. Define an action of the Witt algebra $\W$ on $V_\mu$ by
\begin{equation}\label{E:Gerst:char0:struct:Vmu:action}
w_m . \y^\ell =(\ell-(m+1)\mu)\y^{m+\ell}, \qquad \mbox{for all $m\geq -1$ and $\ell\geq 0$.}
\end{equation}
It can be verified that this indeed defines an action of $\W$ on $V_\mu$, for any $\mu\in\FF$ (for $\mu$ of the form $\frac{\alpha-i}{\alpha-1}$ with $\alpha\geq i+2$ this statement is implied by \eqref{E:Gerst:char0:struct:Si:ss:action}).

The module $V_{\mu}$ is related to the intermediate series modules for the Witt and Virasoro algebras (compare \eqref{E:Gerst:char0:struct:Umu:action}, ahead). Next, we record irreducibility and isomorphism criteria for these modules.

\begin{lemma}\label{L:Gerst:char0:struct:Vmu:irred}
For $\FF$ an arbitrary field of characteristic $0$ and $\mu\in\FF$, let $V_\mu$ be the $\W$-module defined in \eqref{E:Gerst:char0:struct:Vmu:action}. Then:
\begin{enumerate}[label=\textup{(\alph*)}]
\item $V_\mu$ is irreducible if and only if $\mu\neq 0$;\label{L:Gerst:char0:struct:Vmu:irred:1}
\item $V_\mu\cong V_{\mu'}$ if and only if $\mu=\mu'$.\label{L:Gerst:char0:struct:Vmu:irred:2}
\end{enumerate}
\end{lemma}
\begin{proof}
The proof is straightforward, so we just sketch it. First, if $\mu=0$ then $\FF\y^0$ is a submodule of $V_0$, so $V_0$ is reducible. Suppose now that $\mu\neq 0$. Let $X$ be a nonzero submodule of $V_\mu$. Since $w_{-1}^\ell. \y^\ell=\ell !\y^0$, it follows by the usual argument that $\y^0\in X$. Taking into account that $w_m.\y^0=-(m+1)\mu\y^m\in X$ for all $m\geq 0$ and $\mu\neq 0$, we deduce that $X=V_\mu$. Thus $V_\mu$ is irreducible and \ref{L:Gerst:char0:struct:Vmu:irred:1} is proved.

The action of $w_0$ on $V_\mu$ is diagonalizable with eigenvalues $\{\ell-\mu \}_{\ell\geq 0}$, with $-\mu$ being the unique eigenvalue such that $-\mu-1$ is no longer an eigenvalue. Thus the action of $\W$ on $V_\mu$ determines $\mu$, which proves \ref{L:Gerst:char0:struct:Vmu:irred:2}.
\end{proof}

It follows from the above that for all $0\leq i\leq m_h-1$ and all $j$ such that $\alpha_j\geq i+2$, the $\barr{\D}_j\otimes\W$-module $\barr{\D}_j[\y]\subseteq S_i$ is irreducible and it is isomorphic to $\barr{\D}_j\otimes V_{\mu_{ij}}$, where $\mu_{ij}=\frac{\alpha_j-i}{\alpha_j-1}\neq 0$. As the action depends on $i$, it is convenient to introduce $i$ into the notation for this module. Thus, we henceforth denote this module by $\barr{V}_{ij}$:
\begin{equation*}
\barr{V}_{ij}=\barr{\D}_j[\y]\subseteq S_i \quad \mbox{and} \quad \barr{V}_{ij}\cong \barr{\D}_j\otimes V_{\mu_{ij}}, 
\end{equation*}
for all $0\leq i\leq m_h-1$ and $j$ such that $\alpha_j\geq i+2$. Moreover, $\barr{\D}_q\otimes\W$ acts trivially on $\barr{V}_{ij}$ for $q\neq j$, so it follows by Theorem~\ref{T:Gerst:char0:recallBLO3} and \eqref{E:Gerst:char0:struct:ssdecomp} that $\barr{V}_{ij}$ is an irreducible $\hoch^1 (\A)$-submodule of $S_i$ on which both $\mathsf{Z}(\hoch^1(\A))$ and the nilpotent radical $\mathcal N$ of $[\hoch^1 (\A), \hoch^1 (\A)]$ act trivially. As a result of this analysis, we conclude that $S_i$ is a completely reducible $\hoch^1 (\A)$-module with semisimple decomposition (\textit{cf.}\ \eqref{E:Gerst:char0:struct:Si}):
\begin{equation}\label{E:Gerst:char0:struct:Si:ssdecomp}
S_i= \bigoplus_{\alpha_j\geq i+2} \barr{V}_{ij}.
\end{equation}

We summarize these results in the following, which constitutes the main result of this paper.

\begin{thm}\label{T:Gerst:char0:main}
Assume that $\chara(\FF)=0$ and $\A=\A_h$ for $0\neq h\in\FF[x]$. Let $h=\pr_1^{\alpha_1}\cdots \pr_t^{\alpha_t}$ be the decomposition of $h$ into irreducible factors with $0\leq k\leq t$ such that $\alpha_1, \ldots, \alpha_k\geq 2$ and $\alpha_{k+1}= \cdots = \alpha_t=1$. Since $\hoch^2(\A)\neq 0$ if and only if $k\geq 1$, we assume that $k\geq 1$ and set $m_h=\max\{ \alpha_j-1 \mid 1\leq j\leq k\}$. 

The structure of $\hoch^2(\A)$ as Lie module over the Lie algebra $\hoch^1(\A)$ under the Gerstenhaber bracket is as follows:
\begin{enumerate}[label=\textup{(\alph*)}]
\item There is a filtration of length $m_h$ by $\hoch^1(\A)$-submodules \label{T:Gerst:char0:main:1}
\begin{equation*}
\hoch^2(\A)=P_0\supsetneq P_1\supsetneq \cdots \supsetneq P_{m_h-1}\supsetneq P_{m_h}=0. 
\end{equation*} 
\item For each $0\leq i\leq m_h-1$ the factor module $S_i=P_i/P_{i+1}$ is completely reducible with semisimple decomposition
$S_i= \bigoplus_{\alpha_j\geq i+2} \barr{V}_{ij}$, where:\label{T:Gerst:char0:main:2}
\begin{enumerate}[label=\textup{(\roman*)}]
\item The nilpotent radical $\mathsf{Z}(\hoch^1(\A))\oplus \mathcal N$ of $\hoch^1 (\A)$ acts trivially on $S_i$, so $S_i$ becomes a $\left(\barr{\D}_1\otimes\W\right) \oplus\cdots\oplus \left(\barr{\D}_k\otimes\W\right)$-module, where $\barr{D}_j\cong \FF[x]/\pr_j\FF[x]$ and $\W= \spann_\FF\{w_i \mid i \geq -1\}$ is the Witt algebra. \label{T:Gerst:char0:main:21}
\item $\barr{V}_{ij}\cong \barr{\D}_j\otimes V_{\mu_{ij}}$, where $\mu_{ij}=\frac{\alpha_j-i}{\alpha_j-1}$ and the irreducible $\W$-module $V_{\mu}$ is described in \eqref{E:Gerst:char0:struct:Vmu:action}.\label{T:Gerst:char0:main:22}
\item $\barr{\D}_q\otimes\W$ acts trivially on $\barr{V}_{ij}$ for $q\neq j$ and $\barr{\D}_j\otimes\W$ acts on $\barr{V}_{ij}$ via \eqref{E:Gerst:char0:struct:Vmu:action}, under scalar extension.\label{T:Gerst:char0:main:23}
\item $\barr{V}_{ij}\cong \barr{V}_{i'j'}$ as $\hoch^1(\A)$-modules if and only if $(i, j)=(i', j')$.\label{T:Gerst:char0:main:24}
\end{enumerate}
\item $\hoch^2(\A)$ has finite composition length equal to $\sum_{j=1}^k(\alpha_j -1)$, the number of irreducible factors of $\gcd(h, h')$ counted with multiplicity; the compositions factors are $\{ \barr{V}_{ij} \mid 0\leq i\leq m_h-1, \alpha_j\geq i+2 \}$, representing distinct isomorphism classes.\label{T:Gerst:char0:main:3}
\item $\hoch^2(\A)$ is a semisimple $\hoch^1(\A)$-module if and only if $m_h\leq 1$, i.e., if and only if $h$ is not divisible by the cube of any non-constant polynomial.\label{T:Gerst:char0:main:4}
\end{enumerate}
\end{thm}

\begin{remark} It turns out that under the same conditions that ensure that $\hoch^2(\A)$ is semisimple, both
$\hoch^0(\A)$ and $\hoch^1(\A)$ are also semisimple $\hoch^1(\A)$-modules: since $\chara(\FF) = 0$, $\hoch^0(\A)= \FF$
is always simple and by \cite[Cor.\ 5.22\, (ii)]{BLO15ja}, $\hoch^1(\A_h)$ is a direct sum of its center---a sum of trivial modules---and simple Lie ideals.
\end{remark}

\begin{proof}
All of the above statements have been proved, except for \ref{T:Gerst:char0:main:24} and \ref{T:Gerst:char0:main:4}. We start with \ref{T:Gerst:char0:main:24}. If $\barr{V}_{ij}\cong \barr{V}_{i'j'}$ then $\barr{\D}_j\otimes\W$ acts non-trivially on $\barr{V}_{i'j'}$, so $j=j'$, by \ref{T:Gerst:char0:main:23}. Thus, by Lemma~\ref{L:Gerst:char0:struct:Vmu:irred}\ref{L:Gerst:char0:struct:Vmu:irred:2}, $\mu_{ij}=\mu_{i'j}$, which in turn implies $i=i'$. 

For the proof of \ref{T:Gerst:char0:main:4}, if $h$ is not divisible by the cube of any non-constant polynomial then $m_h=1$ and $\hoch^2(\A)=S_0$, which we have seen in \ref{T:Gerst:char0:main:2} is semisimple. Conversely, if $m_h\geq 2$ then there is some $i$ such that $\alpha_i\geq 3$, say $i=1$. By \eqref{E:Gerst:char0:formula:actadgan},
\begin{equation*}
\gb{\ad_{\pr_1 \cdots \pr_k a_1}, \y^0}=-\pr_1 \cdots \pr_k\sum_{i=1}^t \alpha_i\pr_1 \cdots \widehat \pr_i \cdots \pr_t \pr_i'\notin \gcd(h, h')\FF[x],
\end{equation*}
because $\pr_1^2$ divides $\gcd(h, h')$ but it does not divide $\gb{\ad_{\pr_1 \cdots \pr_k a_1}, \y^0}$.
But $\ad_{\pr_1 \cdots \pr_k a_1}\in\mathcal N$ and $\mathcal N$ annihilates all the composition factors of $\hoch^2(\A)$, by \ref{T:Gerst:char0:main:21}, so $\hoch^2(\A)$ cannot be semisimple in this case.
\end{proof}

Before we proceed to illustrate our result with some special cases, we first want to establish a connection between the representations $\barr{V}_{ij}$ and the Virasoro algebra. Recall that the Virasoro algebra is the unique (up to isomorphism) central extension of the full Witt algebra of derivations of $\FF[z^{\pm 1}]$. This Lie algebra is defined as $\mathsf{Vir}=\bigoplus_{i\in\ZZ}\FF .w_i \oplus \FF.c$, where
\begin{equation*}
[c, \mathsf{Vir}]=0 \quad\mbox{and}\quad  \gb{w_m, w_n}=(n-m)w_{m+n}+\delta_{m+n, 0}\frac{m^3-m}{12}c\quad \forall m, n\in\ZZ.
\end{equation*}
Define, for $\mu\in\FF$, the $\mathsf{Vir}$-module $U_\mu=\FF[\y^{\pm 1}]$ with action 
\begin{equation}\label{E:Gerst:char0:struct:Umu:action}
w_m . \y^\ell =(\ell-(m+1)\mu)\y^{m+\ell} \quad \mbox{and}\quad c.\y^\ell=0, \qquad \forall \ell, m\in\ZZ.
\end{equation}
The module $U_\mu$ is an intermediate series module (see \cite{oM92} for details).

The following can be readily checked by the reader:
\begin{enumerate}[label=\textup{(\alph*)}]
\item $\W$ is a Lie subalgebra of $\mathsf{Vir}$.
\item The formula \eqref{E:Gerst:char0:struct:Umu:action} gives a well-defined action of $\mathsf{Vir}$ on $U_\mu$.
\item $V_\mu\subseteq U_\mu$ as $\W$-modules.
\item $U_\mu$ is irreducible as a $\mathsf{Vir}$-module if and only if $\mu\neq0$ and $\mu\neq1$.
\item $U_\mu\cong U_{\mu'}$ as $\mathsf{Vir}$-modules if and only if $\mu=\mu'$.
\end{enumerate}

\subsection{Special cases}\label{SS:Gerst:char0:special:cases}

We end this section with a discussion of some examples of special interest. To avoid trivial cases, in all examples the polynomial $h$ is assumed to be divisible by the square of some non-constant polynomial. We continue to assume that $\chara(\FF)=0$.

\begin{exam}[$h=x^n$]
Let's consider the case where $h$ has a unique irreducible factor. For the sake of simplicity, we will assume that this factor is $x$, that is, $h=x^n$ with $n\ge 2$; the more general case of an irreducible factor of higher degree is entirely analogous. In this case: 
\begin{gather*}
\mathsf{Z}(\hoch^1(\A_{x^n})) = \FF D_{x^{n-1}}, \quad \mbox{where $D_{x^{n-1}}(x)=0$ and $D_{x^{n-1}}(\y)=x^{n-1}$,}\\
\mathcal N =\spann_\FF\{\ad_{x^i a_m} \mid  1\leq i\leq n-2, \ m \geq 0\},\\
[\hoch^1 (\A_{x^n}), \hoch^1 (\A_{x^n})]/\mathcal N \cong \W \ \mbox{(the Witt algebra)},\\
\hoch^2 (\A_{x^n})=\D[\y], \quad \mbox{where $\D=\displaystyle \left( \FF[x]/x^{n-1}\FF[x] \right)$}.
\end{gather*}

For $0\leq i\leq n-1$, let $P_i=x^i\D[\y]$, so that we get the following filtration of $\hoch^1 (\A_{x^n})$-submodules of $\hoch^2 (\A_{x^n})$
\begin{equation*}
\hoch^2(\A_{x^n})=P_0\supsetneq P_1\supsetneq \cdots \supsetneq P_{n-2}\supsetneq P_{n-1}=0. 
\end{equation*} 

Set $S_i=P_i/P_{i+1}\cong \FF[\y]$, for $i\leq n-2$. Then $D_{x^{n-1}}.\hoch^2(\A_{x^n})=0$ and $\mathcal N.P_i\subseteq P_{i+1}$, so $S_i$ is naturally a module for the Witt algebra $\W$, with action 
\begin{equation*}
w_m . \y^\ell =(\ell-(m+1)\frac{n-i}{n-1})\y^{m+\ell}, \qquad \mbox{for all $m\geq -1$ and $\ell\geq 0$.}
\end{equation*}
Thus, $S_i\cong V_{\frac{n-i}{n-1}}$ is simple and the composition factors $\left\{ S_i \right\}_{0\leq i\leq n-2}$ of $\hoch^2(\A_{x^n})$ are pairwise non isomorphic. In particular, $\hoch^2(\A_{x^n})$ has length $n-1$ as a $\hoch^1(\A_{x^n})$-module, with distinct composition factors.
\end{exam}

The next example, a particular case of the previous one, focuses on the Jordan plane.

\begin{exam}[The Jordan plane]
Taking $h=x^2$, we obtain the algebra $\A_{x^2}$, known as the Jordan plane, with homogeneous defining relation $\y x = x\y +x^2$. The description here is:
\begin{equation*}
\hoch^1 (\A_{x^2})=\FF D_{x} \oplus \W \quad \mbox{and}\quad 
\hoch^2 (\A_{x^2})=\FF[\y],
\end{equation*}
where $D_{x}(x)=0$, $D_{x}(\y)=x$ and $\W$ is the Witt algebra.

It follows that $\hoch^2 (\A_{x^2})$ is a simple $\hoch^1 (\A_{x^2})$-module annihilated by $D_{x}$ and such that, as a $\W$-module, $\hoch^2 (\A_{x^2})\cong V_{2}$.

The Lie subalgebra $\FF w_{-1}\oplus\FF w_{0}\oplus\FF w_{1}\subseteq\W$ is isomorphic to $\mathfrak{sl}_2(\FF)$, under the identification $e=w_{-1}$, $h=-2w_0$, $f=-w_1$, where $e=E_{12}$, $f=E_{21}$ and $h=[e, f]$ are the canonical generators of $\mathfrak{sl}_2(\FF)$. The restriction of the $\hoch^1 (\A_{x^2})$-module structure of
$\hoch^2 (\A_{x^2})=\FF[\y]$ to $\mathfrak{sl}_2(\FF)$ is determined by the relations
\begin{equation*}
e.\y^\ell=\ell\y^{\ell-1}, \quad h.\y^\ell=(4-2\ell)\y^\ell, \quad f.\y^\ell=(4-\ell)\y^{\ell+1}, \quad \forall \ell\geq 0. 
\end{equation*}
Whence, it is easy to see that $L(4):=\FF \y^0\oplus\FF \y^1\oplus\FF \y^2\oplus\FF \y^3\oplus\FF \y^4$ is a simple $\mathfrak{sl}_2(\FF)$-submodule of $\hoch^2 (\A_{x^2})$. In fact, $L(4)$ is the simple $\mathfrak{sl}_2(\FF)$-module of highest weight $4$ and the quotient module $\hoch^2 (\A_{x^2})/L(4)\cong M(-6)$ is the irreducible Verma module of highest weight $-6$.

\end{exam}

Our last example deals with the case where $\hoch^2(\A)$ is a semisimple Lie module.

\begin{exam}[$h$ is cube free]
By Theorems \ref{T:Gerst:char0:recallBLO3} and \ref{T:Gerst:char0:main}\,\ref{T:Gerst:char0:main:4}, the following conditions are equivalent:
\begin{itemize}
\item $\hoch^2(\A)$ is a semisimple $\hoch^1(\A)$-module;
\item $\mathcal N=0$;
\item $\hoch^1(\A)$ is a reductive Lie algebra;
\item $h$ is cube free.
\end{itemize}
Here we study the case in which these conditions hold, so the decomposition of $h$ into irreducible factors is of the form $h=\pr_1^2\cdots \pr_k^2\pr_{k+1}\cdots \pr_{t}$, for some $1\leq k\leq t$.
We have
\begin{gather*}
\dim_\FF\mathsf{Z}(\hoch^1(\A))=\degg \pr_1\cdots\pr_{t},\\
\hoch^1 (\A)= \mathsf{Z}(\hoch^1(\A)) \oplus (\barr{D}_1\otimes\W)\oplus \cdots \oplus (\barr{D}_k\otimes\W),\\
\hoch^2 (\A)=\barr{D}_1[\y]\oplus\cdots\oplus\barr{D}_k[\y],
\end{gather*}
where $\barr{D}_j\cong \FF[x]/\pr_j\FF[x]$ and $\W$ is the Witt algebra.

Then, $\mathsf{Z}(\hoch^1(\A))$ acts trivially on $\hoch^2 (\A)$ and $\barr{D}_i\otimes\W$ acts trivially on $\barr{D}_j[\y]$, if $i\neq j$. As a $\barr{D}_j\otimes\W$-module, $\barr{D}_j[\y]\cong \barr{D}_j\otimes V_2$. Thus the irreducible summands of $\hoch^2 (\A)$ are $\left\{\barr{D}_j[\y]\right\}_{1\leq j\leq k}$, they are pairwise non-isomorphic as $\hoch^1 (\A)$-modules and the composition length of $\hoch^2 (\A)$ is $k$.
\end{exam}


\def\cprime{$'$} \def\cprime{$'$} \def\cprime{$'$}

\end{document}